\documentclass[10pt]{article}
\usepackage{indentfirst,latexsym,bm,amsmath,graphicx,amsthm,amssymb}
\usepackage{color}

\usepackage{paralist} 
\usepackage{multirow}
\usepackage{verbatim}
\usepackage{array}
\usepackage{booktabs}

\usepackage{mathrsfs}

\usepackage{graphicx}
\usepackage{epstopdf}
\newtheorem{theorem}{Theorem}
\newtheorem{definition}{Definition}
\newtheorem{lemma}{Lemma}
\newtheorem{remark}{Remark}
\newtheorem{corollary}{Corollary}

\def\C{{\cal{C}}}
\def\U{{\cal{U}}}
\def\G{{\cal{G}}}

\def\rS{{\bar{S}}}
\def\rQ{{\bar{Q}}}
\def\rT{{\bar{T}}}

\def\rD{{\bar{D}}}
\def\rP{{\bar{P}}}
\def\rY{{\bar{Y}}}
\def\rM{{\bar{M}}}
\def\rz{{\bar{z}}}
\def\drD{\dot{{\bar{D}}}}
\def\drP{\dot{{\bar{P}}}}
\def\drT{\dot{{\bar{T}}}}
\def\drY{\dot{{\rY}}}

\renewcommand{\S}{{\cal S}}

\long\def\hq#1{{\color{blue}#1}}

\begin{document}

\title{Fluid Models of Parallel Service Systems under FCFS}

\newcommand*\samethanks[1][\value{footnote}]{\footnotemark[#1]}
\author{
Yuval Nov\thanks{
Department of Statistics,
The University of Haifa,
Mount Carmel 31905, Israel,
\textit{yuval@stat.haifa.ac.il.}
Research supported in part by
Israel Science Foundation Grant 286/13.}
\and
Gideon Weiss\thanks{
Department of Statistics,
The University of Haifa,
Mount Carmel 31905, Israel,
\textit{gweiss@stat.haifa.ac.il.}
Research supported in part by
Israel Science Foundation Grant 286/13.}
\and
Hanqin Zhang\thanks{
Department of Decision Sciences,
School of Business,
National University of Singapore, Singapore,
\textit{bizzhq@nus.edu.sg.}  }
}

\date{ \today }

\maketitle
{\begin{abstract}
We study  deterministic fluid approximations of parallel service systems operating under first come first served policy (FCFS). The condition for complete resource pooling is identified in terms of the system structure and the customer service times.  The static planning linear programming approach (Harrison and Lopez \cite{harrison-lopez:99}) is used to obtain a maximum throughput compatibility tree and to show that FCFS using this compatibility tree is throughput optimal. {We investigate matching rates and show by Hotelling's $T^2$-test and simulation that they are dependent on the service time distribution.}
\end{abstract}

}


\section{Introduction}
\label{sec.introduction}

Parallel service systems are widely used to model service and manufacturing systems.
Such systems have parallel servers $\S=\{s_1,\ldots,s_J\}$ of various skills, a stream of customers of various types $\C=\{c_1,\ldots,c_I\}$, and a bipartite compatibility graph $\G$ where $(s_j,c_i)\in \G$ if server $s_j$ can serve customers of type $c_i$; see Figure \ref{fig.callcenter}.
\begin{figure}[htbp]
   \centering
   \includegraphics[scale=0.45]{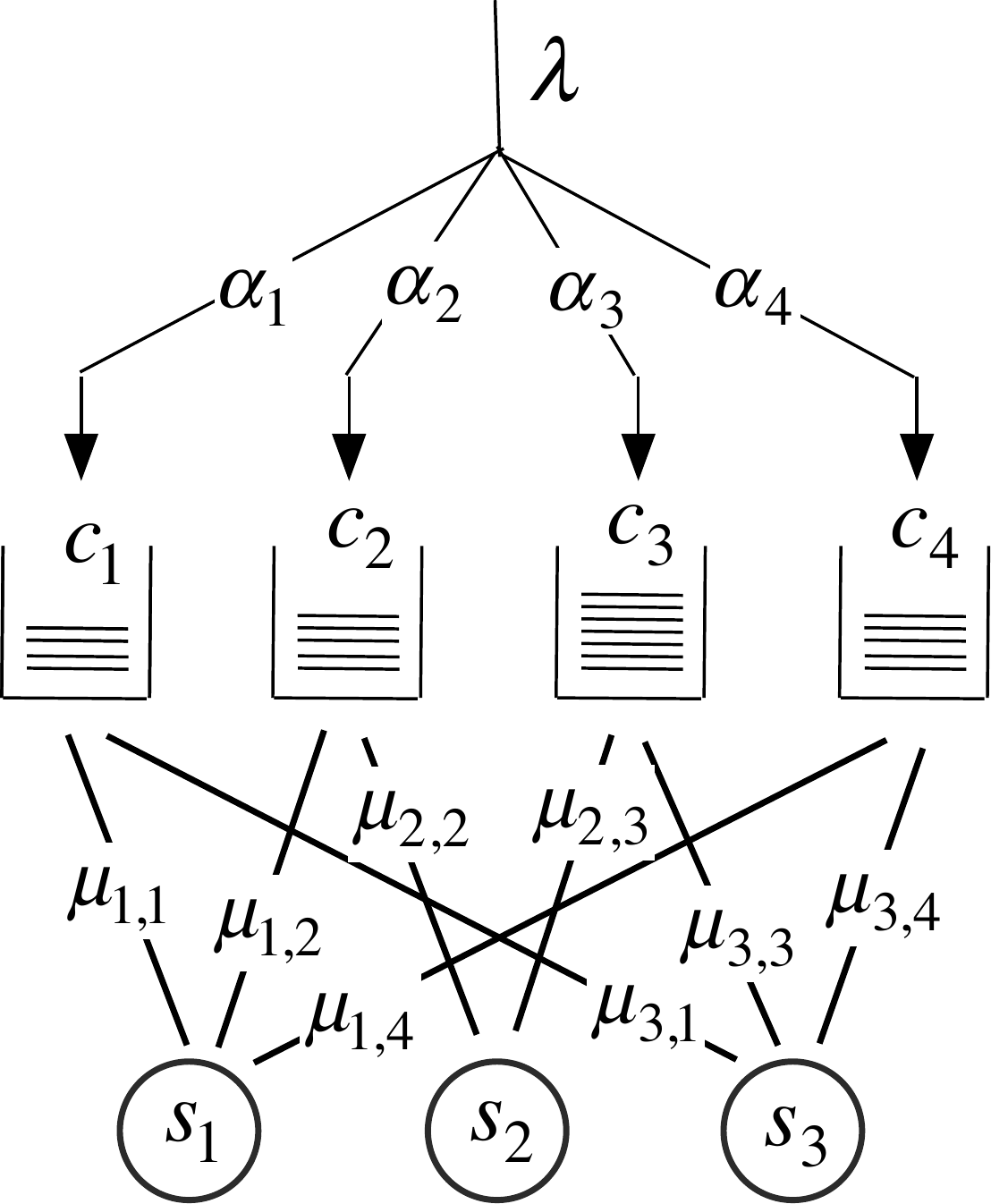}
   \caption{A parallel skilled based service system with 3 servers and 4 customer types}
  \label{fig.callcenter}
 \end{figure}
{In this paper we focus on the behavior of such systems under the policy of first come first served (FCFS), and in particular, on
deterministic fluid approximations for such systems uniformly scaled by time and space.}


It is well known that the policy of FCFS for parallel service systems is not optimal, in that it may waste resources and result in longer waiting times than under other policies.  It is nevertheless very widely used, because it is simple to implement, does not require any knowledge of system parameters, and is fair to customers.  An important property of FCFS is as follows:  Assume that arriving customers have complete information of the system at their arrival, and can choose among the compatible servers which queue to join, and each server uses FCFS for his queue.  In that case, the policy of join the shortest work load (JSW) will  be the Nash equilibrium for customers that wish to minimize their waiting times.  But this policy of JSW is automatically achieved when customers queue up in a single queue and the servers are using FCFS.   FCFS can then serve as a benchmark, and comparison  with other policies will provide an estimate of the ``price of anarchy''.  In particular, performance under FCFS may help in designing the system, e.g., deciding on an improved compatibility graph, and improved service rates.

{Moreover, using FCFS has two purposes:}  under appropriate conditions it introduces resource pooling, i.e., all servers are busy at the same time and act like a combined server,  and it gives the same service level to customers of all different types, i.e., it achieves approximately  global FCFS (as defined by Talreja and Whitt \cite{talreja-whitt:07}).

Our goal in this study is to determine conditions for complete resource pooling, i.e., conditions on the system parameters such that under FCFS all the servers can act as  a combined server providing global FCFS, no matter what the arrival rate is, and to determine the maximal service capacity of the system  in that case.  This maximal service capacity determines stability under any arrival rates, including time-varying arrival rates.  Alternatively, if complete resource pooling fails to hold, we wish to determine whether the servers decompose uniquely to subsets that have complete resource pooling.

The literature on parallel service system is quite voluminous. An incomplete list would include an early study \cite{green:85};  applications to manufacturing and supply chain management \cite{rubino-ata:09, veeger-etman-rooda:08}, applications to call centers and internet service systems \cite{gans-koole-mandelbaum:2003,harchol-balter-etal:99,squillante-etal:01,wallace-whitt:05},  attempts to find optimal policies, mainly for small graph systems \cite{armony-ward:10,armony-ward:13,bell-williams:01,ghamami-ward:13,tezcan-dai:10,williams:00}, heavy traffic and fluid approximations \cite{harrison-lopez:99,harrison-zeevi:05}, and many-server scaling \cite{adan-boon-weiss:14,gurvich-whitt:09,gurvich-whitt:10}. {In view of \cite{gurvich-whitt:10,harrison-lopez:99,harrison-zeevi:05}, establishing fluid approximations is often the first step to solve the optimal dynamic control problem for parallel service systems. {\it Thus, it would be necessary to provide a unified framework of establishing fluid approximations for such system with an arbitrary compatibility graph $($topology$)$.}}

{On the other hand,  to evaluate the utilization of each server  and customer quality of service, one needs to compute matching rates: the fraction of services by server $s_j$ to customers of type $c_i$, out of all services performed by the system.  It is straightforward to see that the matching rates immediately determine resource pooling and maximal service capacity.
Adan and Weiss \cite{adan-weiss:14} discuss  the special case when service rates depend only on {the server}, arrivals are Poisson, and service is exponential, under the policy of FCFS-ALIS (assign longest idle server) and derive a product-form stationary distribution for this system. From the stationary distribution it is possible to calculate matching rates, which in heavy traffic are equal to those obtained  for the FCFS infinite bipartite matching model of \cite{adan-busic-mairesse-weiss:15,adan-weiss:11,caldentey-kaplan-weiss:09}.  The matching rates of the FCFS infinite bipartite matching model reappear in the analysis of parallel service systems with many servers, as demonstrated empirically in \cite{adan-boon-weiss:14}. {\it Motivated by Adan and Weiss} \cite{adan-weiss:14}, {\it we want to see whether the matching rates can be completely determined  from the  first moments of the customers interarrival and service times, in general, or under specific assumptions on the topology of the system
and the interarrival and service distributions.}}

{Furthermore, as observed from some of the above literature, to understand the behavior of parallel service systems one needs to characterize the dynamics of the positions of the $J$ servers in the queue.
 The position dynamics of the $J$ servers can be used to determine  whether resource pooling holds and to calculate the maximal service capacity of the system.  Adan and Weiss \cite{adan-weiss:11,adan-weiss:14} characterize the position dynamics of the servers in the case of server dependent service rates, Poisson arrivals and exponential service times. {\it The natural  question is whether we can determine the fluid trajectories of the positions of the $J$ servers under general assumptions on customer arrivals and service times.}}

{Finally, by the work of Dai \cite{dai:95} for multiclass queueing networks,  fluid approximations not only provide an asymptotic analysis but also verify  stability in the sense of positive Harris recurrence and existence of stationary distributions.
{\it One would like to see whether  fluid approximations can also be used to verify the stability for parallel service systems}.
 Foss and Chernova \cite{foss-chernova:98}  consider parallel service systems under JSW (as well as join shortest queue, JSQ).  They derive conditions for stability  when the service rates depend only on {the servers} and not on the customer types, and also when the service rates depend only on the {customer type} and not on the server.  For the general case, when service rates depend both on the server and  customer type, they produce an example in which stability depends not only on service rates but also on the complete distributions of the service times --- this means that the fluid model is not informative enough to determine stability. {\it Thus it would be interesting to find conditions such that the system stability can be determined by the corresponding fluid approximations}.

 Mainly motivated by the above, in this paper we focus on the following questions
 \begin{itemize}
\item  Establish the deterministic fluid approximations;
\item  Explore  stability conditions for parallel service systems using the fluid model approach;
\item  Obtain  explicit fluid trajectories for the server-dependent (SD) and customer-dependent (CD) processing rates cases;
\item  Find matching rates for parallel service systems with complete bipartite graphs, tree bipartite graphs, or hybrids of those;
\item  Derive a bound on service capacity, and obtain an optimal tree graph for which FCFS achieves the bound and is throughput optimal,
by introducing and solving an LP static planning problem;
\item  Demonstrate by Hotelling's $T^2$-test and simulation of the SD case that matching rates depend on the service time distributions.
\end{itemize}

The rest of the paper is organized as follows. in Section \ref{sec.model} we describe our model, define stochastic processes that describe its dynamics, and define matching rates and resource pooling. Section \ref{sec.fluidlimits} is devoted to the fluid model. We show that fluid limits exist, and derive some fluid model equations that every fluid limit needs to satisfy. The discussion on the stability of parallel service systems and an example of Foss and Chernova \cite{foss-chernova:98} are given in Section \ref{sec.stability}.
 In Section \ref{sec.serverdependent}, we use the fluid model equations to obtain the explicit fluid trajectories for the SD (server-dependent) processing rates case.
The fluid model equations are also used to obtain the explicit fluid trajectories for the CD (customer-dependent) processing rates case in Section \ref{sec.customerdependent}.
 We obtain fluid trajectories for parallel service systems with complete bipartite graphs, tree bipartite graphs, or hybrids of those, by calculating matching rates  in Section \ref{sec.computablematchingrates}.
 We formulate an LP static planning problem (cf. \cite{harrison-lopez:99}) to find a bound on service capacity, and obtain an optimal tree graph for which FCFS achieves the bound and is throughput optimal in Section \ref{sec.optimal}.
 Finally in Section \ref{sec.simulation} we demonstrate by simulation of the SD case, that matching rates depend on the service time distributions, but are very close to the values computed analytically for exponential service times.
}

\section{The stochastic system model}
\label{sec.model}

Given the servers $\{s_1,\ldots,s_J\}$, the customer types $\{c_1,\ldots,,c_I\}$ and the compatibility graph $\G$, the
primitives of the stochastic system consist of a sequence of interarrival times, a sequence of customer types,  and one sequence of processing times for each compatibility link in the graph $\G$.  We assume all these sequences are independent.

We let $a(\ell)$ be the arrival time of the $\ell$th customer and $u(\ell)=a(\ell)-a(\ell-1)$ be the interarrival times, where $\ell=0,\pm 1,\pm 2,\ldots$, and $a(0)\le 0 <a(1)$, and we let $A(t)=\max\{\ell : a(\ell) \le t \} $.  The distribution of $u(\ell)$ is $F$ with mean $1/\lambda$, so that $A(t)$ is a renewal process with rate $\lambda$ (all the fluid model results below continue to hold if we assume only that the arrival process $A(t)$ is stationary and  $A(t)/t \to \lambda$ a.s.).  In particular,  for $s<t$, $A(t)-A(s)$ counts the total number of arrivals in $(s,t]$.  Customer types are i.i.d.,  type $c_i$ has probability $\alpha_{c_i},\,i=1,\ldots,I$, and we let $\xi(\ell)$ be a unit vector of length $I$ such that $\xi_i(\ell) = 1$ if customer $\ell$ is of type $c_i$, for $\ell=0,\pm 1,\pm 2,\ldots$.  The counts of arrivals of customers of each type are then given by
\begin{equation}
\label{eqn.ciarrivals}
A_{c_i}(t) = \left\{ \begin{array}{ll}
\displaystyle \sum_{\ell=1}^{A(t)} \xi_i(\ell), &  t \ge 0, \\
\displaystyle - \sum_{\ell=A(t)+1}^{0} \xi_i(\ell), & t < 0
\end{array}\right.
\end{equation}

We let $v_{s_j,c_i}(0)$ be the remaining service time of server $s_j$ if he is serving a customer of type $c_i$ at time 0, and $v_{s_j,c_i}(0)=0$ otherwise.  We let $v_{s_j,c_i}(k),\,k=1,2,\ldots,$ be the processing time of the $k$th customer of type $c_i$ that server $s_j$ is serving after time 0.  The distribution of $v_{s_j,c_i}(k),\,k=1,2,\ldots,$ is $G_{s_j,c_i}$ with mean $m_{s_j,c_i}$ and rate $\mu_{s_j,c_i}=1/m_{s_j,c_i}$.  We let $X_{s_j,c_i}(t)=\max\{k : \sum_{\ell=0}^k v_{s_j,c_i}(\ell) \le t\}$ count the number of job completions by server $s_j$ when processing customers of type $c_i$ for a total processing time $t$, so that $X_{s_j,c_i}(t)$ is a renewal process of rate  $\mu_{s_j,c_i}$ (all the fluid model results continue to hold if we assume only that the service completion process $X_{s_j,c_i}(t)$ is stationary and  $X_{s_j,c_i}(t)/t \to \mu_{s_j,c_i}$ a.s.).

Service policy is FCFS, i.e., when a server becomes available he will next serve the compatible customer that has been waiting for the longest time.  To complete the service policy description, when a customer arrives and there are several idle compatible servers, then the customer is assigned to the compatible server that has been idle for the longest time, i.e., assign longest idle server, ALIS.

The special case when service rates depend only on the servers, and when arrivals are Poisson and services are exponentially distributed is tractable, and is analyzed in \cite{adan-weiss:14} (see also \cite{adan-visschers-weiss:12}).  Under these assumptions the system can be described by a countable state continuous time Markov chain, and most surprisingly, this Markov chain has a product form stationary distribution.  The following Figure \ref{fig.Ivostate}  describes the state of the Markov chain:  The circles represent the customers in the system ordered from left to right by order of arrivals, $\ell$ busy servers are placed with the customers which they are currently serving, followed by $J-\ell$ idle servers ordered by their idleness times, so that $M_1,\ldots,M_J$ is a permutation of the servers $s_1,\ldots,s_J$.  The  state at time $t$ is defined as $\mathscr{X}(t)=(M_1,n_1,\ldots,M_\ell,n_\ell,M_{\ell+1},\ldots,M_J)$, where $n_j$ counts the number of customers queueing between servers $M_j$ and $M_{j+1}$.  All the customers between $M_j$ and $M_{j+1}$ have been skipped by servers $M_{j+1},\ldots,M_J$ and must therefore be of types in the set $\U(M_1,\ldots,M_j)$ of customer types which are unique customers of $M_1,\ldots,M_j$, {where $\U(M_1,\ldots,M_j)$ is the set of customer types who are not compatible with servers ${\cal S}\setminus \{M_1,\ldots,M_j\}$ (see the definition at the end of this section)}.
\begin{figure}[htbp]
   \centering
   \includegraphics[scale=0.42]{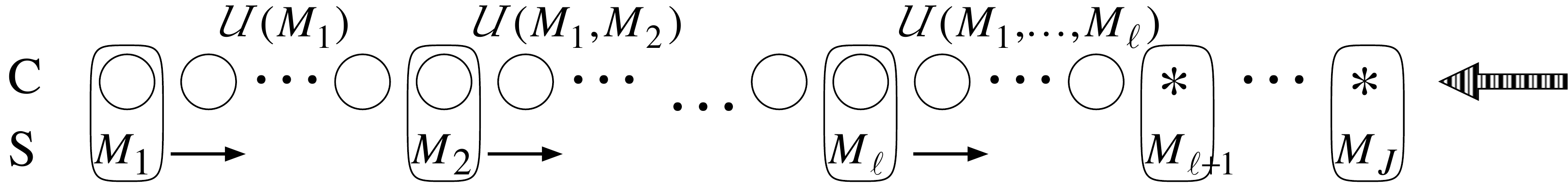}
   \caption{A state for the Markovian FCFS-ALIS parallel skill based system}
   \label{fig.Ivostate}
 \end{figure}
The dynamics are as follows:  Customers arrive from the right, scan the idle servers and join the end of the queue with the first compatible idle server that they find, or without a server if none is available.  When a server completes a service, a customer leaves the system, and the server  moves to the right, scanning the waiting customers until he finds the first compatible customer, or if no such customer is available, he joins the end of the idle servers queue at its left end.  Under the assumption that service rates depend only on the server, Poisson arrivals and exponential services, this is a discrete state continuous time Markov chain.

$\mathscr{X}(t)=(M_1,n_1,\ldots,M_i,n_i,M_{i+1},\ldots,M_J)$ in itself is not a Markov process for our general system, but if we add the remaining time to the next arrival and the remaining times until  service completion for all busy servers, it becomes a Markov process in continuous time with an uncountable state space.  In this paper we consider the dynamics of this more general system, and study its fluid limits.  To describe the dynamics we use a more detailed representation of the system as illustrated in Figure \ref{fig.systemdynamics}, where the system has 3 servers, 3 customer types, and the compatibility graph includes $\G=\{(s_1,c_1),(s_2,c_1),(s_2,c_2),(s_3,c_1),(s_3,c_3)\}$.
\begin{figure}[htbp]
  \centering
   \includegraphics[scale=0.45]{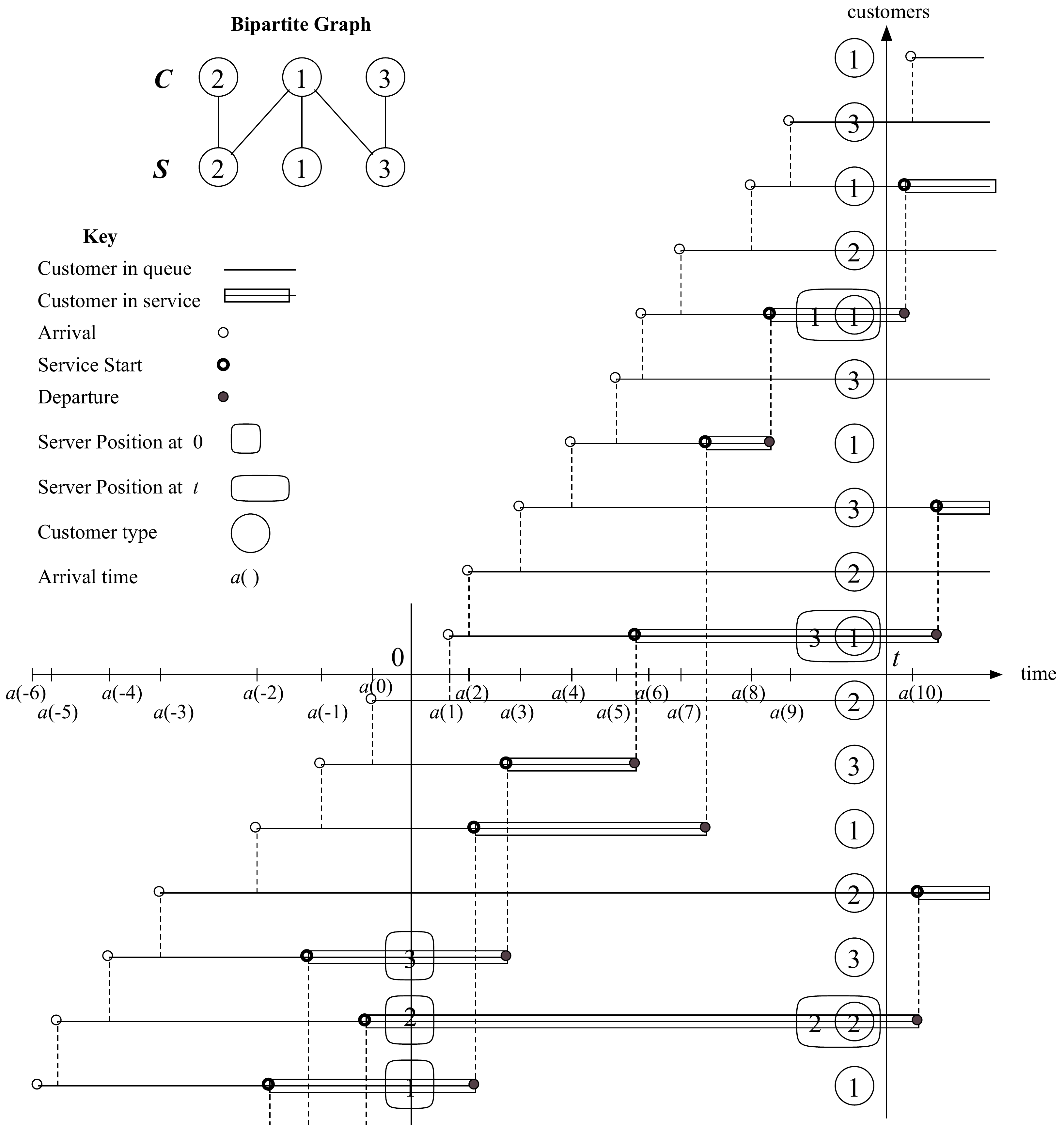}
   \caption{Dynamics of a 3 servers 3 customer type system under FCFS-ALIS}
   \label{fig.systemdynamics}
 \end{figure}
On the horizontal time axis the arrival times of  customers are marked by $a(\ell)$.  On the vertical axis the types of successive customers are listed.  The list includes all the customers, past present and future, starting from the oldest customer that was present at time 0.   For each customer there is a horizontal line starting at his arrival, and ending at his departure, which includes his waiting time and his service time.  With each of the $J$ servers there is a path that describes his whole history, composed of horizontal intervals when he is serving a customer, and vertical intervals that connect the end of service of a customer and the beginning of service of the next customer that he is serving.  A top path describes the counting process of the arrival stream $A(t)$.  When a server is idle he will move together with $A(t)$.

Our working hypothesis is that if we scale time and space uniformly by $n$ and let $n$ increase, we will get fluid limits which will evolve along piecewise linear paths, so that the fluid limits of the picture in Figure \ref{fig.systemdynamics} will look as in Figure \ref{fig.fluiddynamics}.
\begin{figure}[htbp]
   \centering
   \includegraphics[scale=0.45]{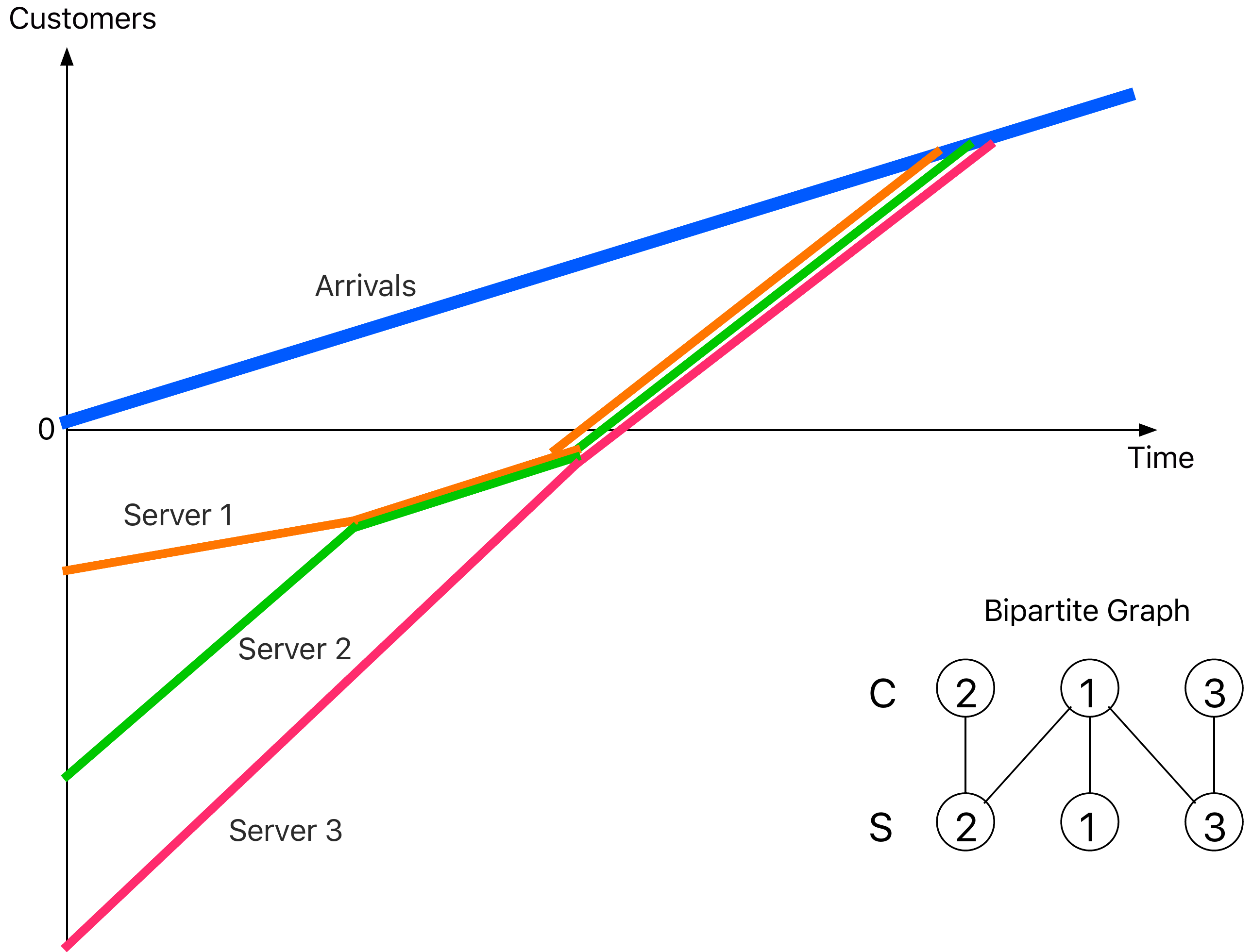}
   \caption{Conjectured Fluid Dynamics  under FCFS-ALIS}
   \label{fig.fluiddynamics}
\end{figure}
Here the horizontal and vertical steps of the paths of servers become increasing straight lines.  The top line records all cumulative fluid arrivals as a function of time, and the arriving fluid is a mixture of the three customer types.  Under the line of server $s_3$  all the arrivals have already departed.  In the area between the lines of servers $s_2,s_3$ fluid of customers of types $c_3$ are still waiting, but types $c_1,c_2$ have already departed.  In the area between the lines of server $s_1$ and $s_2$  only customer fluid of type $c_1$ are  departed, and customers of types $c_2,c_3$ are still waiting.  Finally, between the arrival line and the line of server $s_1$ fluid customers of all three types are still waiting.  In this figure all three lines eventually meet; this is the phenomena of {\em complete resource pooling}.  Furthermore, in this instance the fluid limit is stable, as all fluid is drained and the fluid system is empty from some time onwards.

Assuming that fluid limits move along such straight lines, we are interested in the following questions:
\begin{compactitem}
\item[-]
When $\lambda$ is large, do the lines of all the servers merge eventually?  If so, we say that complete resource pooling holds.
\item[-]
If the lines do not merge, does this define a unique decomposition of the servers?
\item[-]
For what values of $\lambda$ do all the lines eventually merge with $\lambda t$, the top line?  In cases when they merge, we say that the fluid model is stable.
\end{compactitem}
 Complete resource pooling implies under some minor conditions that queues between servers are stable, and stability of the fluid model implies under some minor conditions that the stochastic system is stable.

We introduce some notation:
We denote by $\C(s_j)$ the customer types compatible with $s_j$, referred to as customers of $s_j$, and by $\S(c_i)$ the servers that are compatible with customers of type $c_i$, referred to as the servers of $c_i$.
 For a subset $C \subseteq \C$ of customer types we let $\S(C)= \bigcup_{c_i\in C} \S(c_i)$ denote all the servers of customer types in $C$.  Also, for a subset $S\subseteq \S$ we let $\C(S) = \bigcup_{s_j\in S} \C(s_j)$  denote all the customer types that can be served by some servers in $S$, and we let $\U(S) = \overline{\C(\overline{S})}$ denote the set of customer types which cannot be served by any server outside $S$, that is, the unique customers of $S$.  For a subset $C \subseteq \C$ of customer types we let $\alpha_C=\sum_{c_i\in C} \alpha_{c_i}$.

To describe the dynamics of the system we define the following quantities:
\begin{description}
\item
$P_{s_j}(t)$ is the position of server $s_j$ at time $t$, where we let $P_{s_j}(t)=\ell$ if the server is serving at time $t$ the $\ell$th customer in the sequence of arrivals.  If servers $s_{j_1},\ldots,s_{j_k}$ are idle at time $t$  then their positions are defined as $A(t)+1,\ldots,A(t)+k$, ordered by duration of idleness, with $A(t)+k$ the longest idle.
\item
$Y_j(t)$ is the current $j$th level, where we let $Y_1(t)<\ldots<Y_J(t)$ be the ordered set of the positions of the servers at time $t$.
\item
$M_1(t),\ldots,M_J(t)$ is the random permutation of the servers at time $t$, where we let $P_{M_j(t)}(t)=Y_j(t)$
\item
$T_{s_j,c_i}(t)$ is the cumulative time over $(0,t)$ that server $s_j$ has served customers of type $c_i$.
\end{description}
In this paper we will mainly investigate the processes $Y_j(t),M_j(t),\,j=1,\ldots,J$.  These processes also define the actual queue lengths.  We let $Q_{c_i,j}(t)$ denote the number of customers of type $c_i$ which are waiting between servers $M_j(t)$ and $M_{j+1}(t)$ at time $t$.  These are given by:
\begin{equation}
 Q_{c_i,j}(t) =  \left\{ \begin{array}{ll}  \displaystyle
\sum_{\ell =Y_j(t)+1}^{Y_{j+1}(t)-1} \xi_i (\ell) \, {\mathsf I} \{c_i \in \U(M_1(t),\ldots,M_j(t))\}, & j=1,\ldots,J-1,  \\
\displaystyle  \sum_{\ell =Y_J(t)+1}^{A(t)} \xi_i (\ell),  &  j=J.
\end{array} \right.   \label{eqn.queuedef1}
\end{equation}
where $\mathsf{I}\{\cdot\}$ is the indicator function.

Let $U(t)$ be the remaining time at time $t$ until next arrival, $V_{s_j,c_i}(t)$ be the remaining processing time of $c_i$ by $s_j$ if $s_j$ is processing a type $c_i$ customer at time $t$, and $V_{s_j,c_i}(t)=0$ otherwise.  The initial state of the system is given by $A(0)=0$, $P_{s_j}(0)$, $U(0)=a(1)$, and $V_{s_j,c_i}(0)=v_{s_j,c_i}(0)$ (note that $P_{s_j}(0) < 0$).   Both $\mathscr{Y}(t)=\big(A(t),P_{s_j}(t),U(t),V_{s_j,c_i}(t)\big)$, and
 $\mathscr{Z}(t)=\big(M_j(t),Q_{c_i,j}(t),U(t),V_{s_j,c_i}(t)\big)$ are Markov processes.   The former is always transient, as $A(t),P_{s_j}(t)$ are non-decreasing with $t$.  The latter may be stable, and we say that the queueing system is stable (ergodic) if   $\mathscr{Z}(t)$ is positive Harris recurrent (ergodic).


\section{Fluid limits and fluid model equations}
\label{sec.fluidlimits}

To study the fluid limits of the system we consider a sequence of systems defined over the same probability space, indexed by $n=1,2,\ldots$, and study their fluid scaling.  All the systems in the sequence share the same stochastic sequences $A(t),\xi(\ell),X_{s_j,c_i}(t)$, but they differ in their initial conditions:  We let $P_{s_j}^n(0),\,j=1,\ldots,J$ be the initial positions of the servers in the $n$th system.
We denote quantities of the $n$th system which are not common to all systems by the superscript $n$.  We obtain the fluid scaling for the sequence of systems by scaling time and space of the $n$th system by $n$. For any function $z^n(t)$ we define the fluid scaling as $\rz^n(t)=\frac{1}{n} z^n(nt)$.

Consider sample paths $\omega\in \Omega$ of the sequence of systems, and consider one of the processes, say $z^n(t,\omega)$.  If $\rz^r(t,\omega)=\frac{1}{r} z^r(rt,\omega) \to \rz(t)$ uniformly on compacts (u.o.c.) when $r\to \infty$, for some $\omega$ and for some subsequence $r$ of $n=1,2,\ldots$, where $\rz(t)$ is a deterministic function of $t$, then we say that $\rz(t)$ is a fluid limit of
$z^n(\cdot,\cdot)$.

To obtain the fluid dynamics of our system we need to assume that the following holds:
\begin{eqnarray}
&& \lim_{n\to\infty} \rP^n_{s_j}(0) = \lim_{n\to\infty} \frac{P^n_{s_j}(0)}{n} = \rP_{s_j}(0) \le 0.
  \label{eqn.fluidinitiaipositionlimit} \\
&&  {\lim_{n\to\infty} \frac{U^n(0)}{n} =  0,  \qquad \lim_{n\to\infty} \frac{V^n_{s_j,c_i}(0)}{n} =  0.}
 \label{eqn.initialremainderlimits}
\end{eqnarray}
and
\begin{eqnarray}
\label{eqn.fsllnfluidlimit}
&&\lim_{n\rightarrow \infty}  \frac{1}{n} A(nt,\omega)=\lambda t \ \ \ \ \mbox{u.o.c., a.s. } \nonumber \\
&&\lim_{n\rightarrow \infty}  \frac{1}{n} A_{c_i}(nt,\omega)=\lambda \alpha_{c_i} t \ \ \ \ \mbox{u.o.c., a.s. } \\
&&\lim_{n\rightarrow \infty}  \frac{1}{n} X_{s_j,c_i}(nt,\omega)=\mu_{s_j,c_i}t \ \ \ \ \mbox{u.o.c., a.s. } \nonumber \end{eqnarray}
We assume throughout that (\ref{eqn.fluidinitiaipositionlimit}) and (\ref{eqn.initialremainderlimits}) hold.
Assumptions (\ref{eqn.fsllnfluidlimit}) hold for our system  by the functional strong law of large numbers, since we assume renewal arrivals, i.i.d. customer types, and renewal service times.
We exclude the set of measure zero where (\ref{eqn.fsllnfluidlimit}) fails to hold. {Let $T^n_{s_j,c_i}(t)$ be the cumulative service time of customer type $c_i$ provided by
server $s_j$ over time interval $[0,t]$. The following theorem proves the existence of fluid limits.}

\begin{theorem}
\label{thm.existence}
Fluid limits  for $\rT^n_{s_j,c_i}(t,\omega), \rP^n_{s_j}(t,\omega), \rY^n_j(t,\omega),\rQ^n_{c_i,j}(t)$  exist almost surely for every $\omega$, and they are almost surely Lipschitz continuous for every $t>0$.
\end{theorem}
\begin{proof}
Consider first $T^n_{s_j,c_i}(t,\omega)$.  We have for all $\omega$ that $T^n_{s_j,c_i}(t,\omega)-T^n_{s_j,c_i}(s,\omega) \le t-s$ for all $s<t$, and so for all $\omega$ and every $n$ also $\rT^n_{s_j,c_i}(t,\omega)-\rT^n_{s_j,c_i}(s,\omega) \le t-s$, so the sequence is equicontinuous and uniformly bounded on every compact interval, for every $\omega$.  Fix $\omega$.  By Arzela-Ascoli theorem, for every compact interval there exists a subsequence $r$ of $n$ such that $\rT^r_{s_j,c_i}(t,\omega)$ converges to some $\rT_{s_j,c_i}(t)$ as $r\to \infty$ uniformly on the  interval.  It is then possible to choose a further subsequence that will converge uniformly on all compacts.  Furthermore, all $\rT^n_{s_j,c_i}(t,\omega)$ are Lipschitz continuous for every $\omega$, and hence so is every fluid limit $\rT_{s_j,c_i}(t)$.

The main part of the proof is to show the existence  of fluid limits for $P^n_{s_j}(t,\omega)$.
The functions $P^n_{s_j}(t,\omega)$ are non-decreasing in $t$, and $P^n_{s_j}(0,\omega) \le P^n_{s_j}(t,\omega) \le A(t,\omega) + J$, and hence for any $\epsilon>0$ and large enough $n$,
$\rP_{s_j}(0)-\epsilon \le \rP^n_{s_j}(t,\omega) \le \lambda t + \epsilon$, so $\rP^n_{s_j}(t,\omega)$ are non-decreasing and uniformly bounded at each $t$ for all $n$.  Hence we can find a subsequence $r$ such that $\rP^r_{s_j}(t,\omega) \to \rP_{s_j}(t)$ as $r\to\infty$ for all rational $t$, and we have that $\rP_{s_j}(t)$ is non-decreasing on all rationals, and we can then extend its definition to all real $t$.  If we can show that $\rP_{s_j}(t)$ is continuous, then by Lemma 4.1 of Dai \cite{dai:95}  we will have that $\rP^r_{s_j}(t,\omega) \to \rP_{s_j}(t)$ uniformly on compacts.  We will show that $\rP_{s_j}(t)$ is in fact Lipschitz continuous for $t>0$.

We note that $\rP_{s_j}(t)$ may be discontinuous at $0$.  Consider the limiting $\rP(0),\rY(0),\rM(0)$, and assume the following: (i) $\rY_k(0) = \rP_{s_j}(0)$ (ii) $\C(s_j) \subseteq \C\big(\rM_{k+1}(0),\ldots,\rM_J(0)\big)$  (iii)  $\rY_k(0) < \rY_{k+1}(0)$.
Denote $v_{s_j}(0,\omega) = \max_{c_i\in \C(s_j)}  v_{s_j,c_i}(0,\omega) > 0$.
Then we have that $P^n_{s_j}\big(v^n_{s_j}(0,\omega)\big) > Y^n_{k+1}(0)$, and so we have:
\[
\lim_{n\to \infty} \rP^n_{s_j}(0,\omega) = \rY_k(0)  \quad \mbox{while} \quad
\rP_{s_j}(0+) \ge \liminf_{n\to\infty} \frac{1}{n} P^n_{s_j}\big(\frac{1}{n} v^n_{s_j}(0,\omega)\big) \ge \rY_{k+1}(0)>\rY_k(0).
\]

Consider now $v^n_{s_j}(0,\omega)<t_0<t_1$.
Let $c_i$ be the type of the customer that $s_j$ is serving at time $t_0$,
let $w_{s_j}(t_0,\omega)$ be elapsed  time of this customer, and let $t = t_0- w_{s_j}(t_0,\omega)$ be the time at which the service of this customer started.  At time $t$, by FCFS, all customers of type $c_i$ in positions $> P^n_{s_j} (t,\omega)$ have not yet started service.  During the time period $(t,t_1)$, server $s_j$ is processing customers of type $c_i$ as well as customers of types $c_l \in \C(s_j),\, c_l\ne c_i$.  Hence it may only process at most   $X_{s_j,c_i} (T^n_{s_j,c_i}(t_1)) - X_{s_j,c_i} (T^n_{s_j,c_i}(t))$ customers of type $c_i$.   During the time period $(t,t_1)$, other servers $s_k \in \S(c_i),\,s_k\ne s_j$ may also be processing customers of type $c_i$.  Therefore the total number of customers of types $c_i$ that may be processed in the time period $(t,t_1)$ cannot exceed
$ \sum_{k\in \S(c_i)} (X_{s_k,c_i} (T^n_{s_k,c_i}(t_1)) - X_{s_k,c_i} (T^n_{s_k,c_i}(t)))$.

We repeat the argument of the last paragraph for the scaled processes. {Consider  any $0<t_0<t_1$ and $n$ large enough that $v^n_{s_j}(0)< n t_0$ almost surely by (\ref{eqn.initialremainderlimits})}.  Assume server $s_j$ is working on job type $c_i$ at time $n t_0$, with elapsed time  $w^n_{s_j}(n t_0,\omega)$, and let $n t=n t_0 - w^n_{s_j}(n t_0,\omega)$ be the time that processing of this job started. Then:
\begin{eqnarray*}
&& \frac{1}{n} \Big(P^n_{s_j} (n t_1,\omega)-P^n_{s_j} (n t_0,\omega)\Big)\\
&& \ \ \  =
\frac{1}{n} \Big(P^n_{s_j} (n t_1,\omega)-P^n_{s_j} (n t,\omega)\Big)   \\
&& \quad = \frac{ \sum_{k=1}^I  \sum_{\ell = P^n_{s_j} (n t,\omega)}^{P^n_{s_j} (n t_1,\omega)}  \xi_k(\ell)  }
{ \sum_{\ell = P^n_{s_j} (n t,\omega)}^{P^n_{s_j} (n t_1,\omega)}  \xi_i(\ell)  }\;
\frac{1}{n} \sum_{\ell = P^n_{s_j} (n t,\omega)}^{P^n_{s_j} (n t_1,\omega)}  \xi_i(\ell) \\
&&  \quad \le \frac{ \sum_{k=1}^I  \sum_{\ell = P^n_{s_j} (n t,\omega)}^{P^n_{s_j} (n t_1,\omega)}  \xi_k(\ell)  }
{ \sum_{\ell = P^n_{s_j} (n t,\omega)}^{P^n_{s_j} (n t_1,\omega)}  \xi_i(\ell)  }\;
\frac{1}{n} \sum_{s_k\in \S(c_i)}
\Big( X_{s_k,c_i} (T^n_{s_k,c_i}(n t_1)) - X_{s_k,c_i} (T^n_{s_k,c_i}(n t)) \Big)
\end{eqnarray*}
Going to the limit, we take a subsequence $r$ for which convergence of $\frac{1}{r} P^r_{s_j} (r t,\omega)$ holds.  In the case that {$\lim_{r\to\infty} (P^r_{s_j} (r t_1,\omega) - P^r_{s_j} (r t_0,\omega))<\infty$},
we have $ \rP_{s_j} ( t_1)-\rP_{s_j} (t_0) = 0$.
Otherwise we now consider the above inequality for all $c_i$. We have that as $r\to \infty$:
{
\begin{eqnarray*}
 \rP_{s_j} ( t_1)-\rP_{s_j} (t_0)
&\le&   \max_{c_i\in \C} \left[  \lim_{r\to\infty} \left\{ \frac{ \sum_{k=1}^I  \sum_{\ell = P^r_{s_j} (r t,\omega)}^{P^r_{s_j} (r t_1,\omega)}  \xi_k(\ell)  }
{ \sum_{\ell = P^r_{s_j} (r t,\omega)}^{P^r_{s_j} (r t_1,\omega)}  \xi_i(\ell)  } \right\} \right.\\
&& \left. \qquad \qquad \times \lim_{r\to\infty} \left\{ \frac{1}{r} \sum_{s_k\in \S(c_i)}
\Big( X_{s_k,c_i} (T^r_{s_k,c_i}(r t_1)) - X_{s_k,c_i} (T^r_{s_k,c_i}(r t)) \Big) \right\} \right] \\
&=&  \max_{c_i\in \C}  \Big[  \frac{1}{\alpha_{c_i}}  \sum_{s_k\in \S(c_i)}
\mu_{s_k,c_i}  \Big( \rT_{s_k,c_i}(t_1) - \rT_{s_k,c_i}(t)   \Big)  \Big]   \\
&\le&  \max_{c_i\in \C}  \Big[  \frac{1}{\alpha_{c_i}}  \sum_{s_k\in \S(c_i)}
\mu_{s_k,c_i} \Big] \Big( t_1 - t_0 + \lim_{r\to \infty}  \frac{ w^r_{s_j}(r t_0,\omega) }{ r } \Big)\\
&=&  \max_{c_i\in \C}  \Big[  \frac{1}{\alpha_{c_i}}  \sum_{s_k\in \S(c_i)}
\mu_{s_k,c_i}    \Big]  ( t_1 - t_0).
\end{eqnarray*}}
The last equality holds because $\max_{1\le \ell \le n} V^n_{s_j,c_i}(\ell) \big/ n \to 0$ as $n\to\infty$ for all $s_j,c_i$ a.s. for all $\omega$.
We have therefore that the fluid limits $\rP_{s_j} ( t)$ are Lipschitz continuous with constant
$\max_{c_i\in \C}  \left[  \frac{1}{\alpha_{c_i}}  \sum_{k\in \S(c_i)}
\mu_{s_k,c_i} \right]$.

We can now use subsequences of subsequences to obtain that $\rP^r_{s_j}(t,\omega) \to \rP_{s_j}(t)$ u.o.c. for all $s_j$ as $r\to\infty$.
For this subsequence we then have that $\rY_j^r(t,\omega)\to \rY_j(t)$ which are the ordered values of $\rP_{s_j}(t)$.

Finally, from (\ref{eqn.queuedef1}) we obtain that for almost all $\omega$ there is some subsequence $r$ such that as $r\to\infty$:
\begin{equation}
\label{eqn.fluidqueues}
\rQ^r_{c_i,j}(t,\omega) \to
\left\{
\begin{array}{ll}
 \alpha_{c_i} (\rY_{j+1}(t)-\rY_j(t)),  & c_i\in \U(M_1(t),\ldots,M_j(t)),\;  j=1,\ldots,J-1,\\
 \\
 \alpha_{c_i} (\lambda t - \rY_J(t)), &  c_i \in \C, \quad  j=J.
\end{array}
\right.
\end{equation}
\end{proof}

We now know that almost surely for all $\omega$ there exist subsequences which lead to fluid limits that are Lipschitz continuous for all $t>0$.  We also assume that for these subsequences (\ref{eqn.fluidinitiaipositionlimit})-(\ref{eqn.fsllnfluidlimit}) hold by excluding a set of measure zero.  Since the fluid limits are Lipschitz continuous they are absolutely continuous and hence possess derivatives almost everywhere, and are integrals of their derivatives.  We shall call times $t$ at which derivatives of  fluid limits exist regular times.  We will use $\dot{z}(t)$ to denote $\frac{d}{dt}z(t)$ for all fluid limits, for all regular $t$.  We now wish to derive equations which all fluid limits must satisfy almost surely.

By definition, for every $n$, at every time $t$, $\rP^n_{s_j}(t,\omega)$ for $s_1,\ldots,s_J$ are all different, so that
\[
\rY^n_1(t,\omega)=\frac{1}{n}P^n_{M^n_1(nt,\omega)}(nt,\omega) <
\, \cdots \,< \rY^n_J(t,\omega)=\frac{1}{n}P^n_{M^n_J(nt,\omega)}(nt,\omega)
\]
However, for the fluid limits we only have that $\rY_1(t)\le \rY_2(t) \le \cdots \le \rY_J(t)$.  As a result the fluid limits  no longer define a unique permutation of the servers at time $t$, and instead we have an ordered partition of $s_1,\ldots,s_J$.  For concreteness we {order $\rP_{\rM_1(t)}(t),\ldots,\rP_{\rM_J(t)}(t)$ so that $\rP_{\rM_j(t)}(t)<\rP_{\rM_{j+1}(t)}(t)$ or $\rP_{\rM_j(t)}(t) = \rP_{\rM_{j+1}(t)}(t)$} and $\rM_j(t)<\rM_{j+1}(t)$.
We  define the fluid ordered partition $\bar \S(t) =  \big(\bar S_1(t),\ldots,\bar S_L(t)\big)$ as follows:
\begin{eqnarray}
\label{eqn.partition1}
&\big(\bar S_1(t),\ldots,\bar S_L(t)\big) \mbox{ is a partition of $\S$},& \nonumber \\
&M_j,M_{j'} \in \bar S_\ell(t) \Rightarrow \bar P_{M_j}(t)= \bar P_{M_{j'}}(t),&\\
&M_j \in \bar S_\ell(t) \ \mbox{and} \ M_{j'} \in \bar S_{\ell+1}(t) \Rightarrow  \bar P_{M_j}(t) < \bar P_{M_{j'}}(t).&  \nonumber
\end{eqnarray}
Note that $\bar{\S}(t),\rM(t)$ are limits at time $nt$ when $n\to\infty$, but they are not scaled in space, since they are discrete {and finite}.  We introduce the notation $\rY_S(t),\drY_S(t)$ to denote the common value of $\rP_{M_j}(t),\drP_{M_j}(t),\,M_j\in S$.

We now have the following theorem on the dynamics of the fluid model.  We use the convention that $\mu_{s_j,c_i}=\drT_{s_j,c_i}=0$ for $(s_j,c_i)\not\in \G$.
 \begin{theorem}
\label{thm.generaljointserversdynamics}
Consider a fluid limit in which servers at levels $k,\ldots,l$ move together for a while, i.e., $\rY_{k-1}(\tau)<\rY_k(\tau)=\cdots=\rY_l(\tau)<\rY_{l+1}(\tau)$
(or if $l=J$, $\rY_l(\tau)<\lambda  \tau$),  for some $k \le l$  and for all $s < \tau < t$.  {Let $\bar{S}(\tau)=(S'(\tau),\{M_k,\ldots,M_l\},S^{\prime\prime}(\tau))$ for the same range of $\tau$, where $S'(\tau)$ and $S^{\prime\prime}(\tau)$ are the subsets of servers preceding and succeeding $M_k,\ldots,M_l$ (the sets $S'(\tau)$ and $S^{\prime\prime}(\tau)$} may themselves consist of a further partition,  but this is irrelevant here).    Then a.s. all fluid limits at $s<\tau<t$ must satisfy the following equations:
 \begin{equation}
 \label{eqn.busyrate}
  \sum_{c_i\in\C(M_j)\backslash \C(M_{l+1},\ldots,M_J)}   \dot\rT_{M_j,c_i}(\tau) = 1 \quad j=k,\ldots,l, \\
\end{equation}
 \begin{equation}
\label{eqn.gengetherfluidB}
 \drY_k(\tau)=\cdots =\drY_l(\tau)   = \frac{1}{\alpha_{c_i}}  \sum_{j=k}^l \mu_{M_j,c_i}  \dot\rT_{M_j,c_i}(\tau),
 \qquad c_i\in \C(M_k,\ldots,M_l)\backslash \C(M_{l+1},\ldots,M_J).
\end{equation}
\end{theorem}
\begin{proof}
Consider a fluid limit of all the processes obtained for some $\omega$ and subsequence $r$, for which the assumptions of the theorem hold.

By the continuity of $\rP_{s_j}(\tau)$  the sets  $S'(\tau),\{\rM_k(\tau),\ldots,\rM_l(\tau)\},S''$ are constant for all $s<\tau<t$, and $\rS(\tau)$ is well defined.
If $\rY_j(\tau) <\rY_{l+1}(\tau) \le \lambda \tau,\, s<\tau<t,\, j=k,\ldots,l$ then for $r$ large enough $Y^r_j(r \tau,\omega)< Y^r_{l+1}(r \tau,\omega) < A (r \tau,\omega),\, \tau\in (s,t),\, j=k,\ldots,l$, which implies that $M^r_k,\ldots,M^r_l=M_k,\ldots,M_l$ are the same for all $r$ large enough,  that
all the servers $M_k,\ldots,M_l$ are busy all the time between  $(rs,rt)$, and  the types of customers which $M_j$ will be serving will be $c_i\in\C(M_j)\backslash \C(M_{l+1},\ldots,M_J)$.   It follows that
\[
 \sum_{c_i\in\C(M_j)\backslash \C(M_{l+1},\ldots,M_J)}  \frac{1}{r} \big( T^r_{M_j,c_i}(rt,\omega) - T^r_{M_j,c_i}(rs,\omega) \big)
= t-s
\]
{and (\ref{eqn.busyrate}) follows. For the same $s,t$ and large enough $r$, we have:}
\[
Y^r_k(rs,\omega)=\min_{k\le j \le l} P^r_{M_j}(rs,\omega),\quad
Y^r_l(rs,\omega)=\max_{k\le j \le l} P^r_{M_j}(rs,\omega),\quad
\]
\[
Y^r_k(rt,\omega)=\min_{k\le j \le l} P^r_{M_j}(rt,\omega),\quad
Y^r_l(rt,\omega)=\max_{k\le j \le l} P^r_{M_j}(rt,\omega).\quad
\]
Consider for $c_i \in C(M_k,\ldots,M_l)\backslash\C(M_{l+1},\ldots,M_J)$ the two counts:
\[
N^r_1(\omega) = \sum_{\ell =Y^r_l(rs,\omega)+1}^{Y^r_k(rt,\omega)-1} \xi_i(\ell,\omega),  \qquad
N^r_2(\omega) = \sum_{\ell =Y^r_k(rs,\omega)}^{Y^r_l(rt,\omega)} \xi_i(\ell,\omega)
\]
These count customers of type $c_i$ which are associated with the time interval $(rs,rt)$:  every customer of type $c_i$ which appears in the first count has started service and finished service within the time period
 $(rs,rt)$.
 The second count includes all the customers of type $c_i$ which have departed in the time interval $(rs,rt)$, including some that started processing at an earlier time, and also those which have started service and not departed yet.

Compare this to
{\[
N^r_3(\omega) = \sum_{j=k}^l\Big( X_{M_j,c_i}\big(T^r_{M_j,c_i}(rt,\omega),\omega\big)
 - X_{M_j,c_i}\big(T^r_{M_j,c_i}(rs,\omega),\omega\big)\Big),
\]
}which counts all the service completions of jobs of type $c_i$, served by one of the servers $M_k,\ldots,M_l$, during the time interval $(rs,rt)$ (recall that $T^r_{M_j,c_i}(rt)-T^r_{M_j,c_i}(rs)$ is the total time that server $M_j$ is processing customers of type $c_i$ within the time interval $(rs,rt)$).
We have that $N^r_2(\omega)\ge N^r_3(\omega) \ge N^r_1(\omega)$.

However,
\[
\lim_{r\to\infty} \frac{1}{r} N^r_1(\omega) = \lim_{r\to\infty} \frac{1}{r} N^r_2(\omega)  =
\alpha_{c_i} (\rY_k(t) - \rY_k(s))=\cdots = \alpha_{c_i} (\rY_l(t) - \rY_l(s)),
\]
while
\[
\lim_{r\to\infty} \frac{1}{r} N^r_3(\omega) =
\sum_{j=k}^l \mu_{M_j,c_i}  \big( \rT_{M_j,c_i}(t) - \rT_{M_j,c_i}(s)\big),
\]
and (\ref{eqn.gengetherfluidB}) follows.
\end{proof}

\begin{corollary}
If  $\rY_{j-1}(t)<\rY_j(t)<\rY_{j+1}(t)$ then
\begin{equation}
\label{eqn.singlespeed}
\drY_j(t) = \left( \sum_{c_i\in\C(M_j)\backslash \C(M_{j+1},\ldots,M_J)} \alpha_{c_i} m_{M_j,c_i} \right)^{-1}
\end{equation}
\end{corollary}
\begin{proof}
From (\ref{eqn.gengetherfluidB}) we have that $m_{M_j,c_i} \alpha_{c_i} \drY_j(t) =  \drT_{M_j,c_i}(t)$, and summing over $c_i \in \C(M_j)\backslash \C(M_{l+1},\ldots,M_J)$ and using (\ref{eqn.busyrate}) we obtain (\ref{eqn.singlespeed}).
\end{proof}

We refer to equations (\ref{eqn.busyrate})--(\ref{eqn.singlespeed}) as fluid model equations.


\section{Stability}
\label{sec.stability}
We are interested in verifying the following  properties of fluid limits:
\begin{definition}
Denote $|\rP(0)| = - \sum_{j=1}^J \rP_{s_j}(0)$.

{\rm (i)} We say that the fluid model is stable if starting from any fixed \hq{$|\rP(0)|=1$},  there exists $t_0$ such that for almost surely every fluid limit $\lambda t - \rY_1(t) = 0$ for all $t>t_0$.

{\rm (ii)}   We say that the fluid model has complete resource pooling if for all values of $\lambda$, starting from any fixed $|\rP(0)|$, there exists $t_0$ such that for almost surely  every fluid limit $\rY_J(t)-\rY_1(t) = 0$
for all $t>t_0$.

{\rm (iii)}  We say that the fluid model has complete weak resource pooling if for all values of $\lambda$, starting from any fixed $|\rP(0)|$, there exists $t_0$ such that for almost surely  every fluid limit  $\drY_1(t)=\cdots = \drY_J(t)$
\end{definition}

Complete weak resource pooling is the situation in which in the limit, all servers move eventually at the same rate, but not together.  Graphically, this means that the straight lines denoting their limiting paths become parallel, but may never merge.

\begin{definition}
We say that the fluid model of a system under some given policy is maximum throughput with processing rate $\mu^*$ if the fluid model for the given policy is stable for all $\lambda < \mu^*$, and if the fluid model of the  system is unstable for all $\lambda > \mu^*$ under every policy.
\end{definition}

Complete resource pooling and stability of the fluid limits and fluid model have far-reaching consequences for the stochastic system if some technical conditions are satisfied.   In particular, in the following three theorems we will make the technical assumption that in the state space of the Markov processes considered, every bounded set of states is unifromly small.  For definition of  uniformly small sets of states in a Markov process, see Bramson \cite{bramson:08} or Meyn and Tweedie \cite{meyn-tweedie:93}.

\begin{theorem}
\label{thm.ergodic}
Consider the Markov process $\mathscr{Z}(t)$ and define  {$\sum_{j=1}^J \sum_{i=1}^I (Q_{c_i,j}(t)+V_{s_j,c_i}(t))+U(t)$} as its norm.  Assume that every bounded set of states is uniformly small.
If the fluid model of the system is stable then the process $\mathscr{Z}(t)$ is ergodic, i.e. it possesses a stationary distribution, and the distribution of its state at time $t$ converges to this stationary distribution as $t\to\infty$.
\end{theorem}
\begin{proof}
This follows immediately from the fundamental theorem of Dai \cite{dai:95} and its extension in the monograph of Bramson \cite{bramson:08}.
\end{proof}
\begin{theorem}
\label{thm.pooledergodic}
For the Markov process $\mathscr{Z}(t)$  define  {$\sum_{j=1}^J \sum_{i=1}^I (Q_{c_i,j}(t)+V_{s_j,c_i}(t))+U(t)$}  as  norm, and  assume that every bounded set of states is uniformly small as in Theorem {\rm \ref{thm.ergodic}}.    Consider the  process $\mathscr{Z}^0(t)$ obtained from $\mathscr{Z}(t)$ by the exclusion of the components $U(t),Q_{c_i,J}(t)$.
If complete resource pooling of the fluid model holds, and if $\rY_J(t) < \lambda t$, then there exists a  measure $\nu^0$ on the state space of $\mathscr{Z}^0(\cdot)$ such that as $t\to \infty$ the distribution of   $\mathscr{Z}^0(t)$ converges to $\nu^0$.
\end{theorem}
\begin{proof}
Consider the same system with infinite supply of work, i.e., there is always a queue of customers waiting behind the most advanced server, of types $c_i\in \C$ i.i.d. with probabilities $\alpha_{c_i}$.  Then in this new system $\mathscr{Z}^0(t)$ is a Markov process, and with the norm {$\sum_{i=1}^I\Big(\sum_{j=1}^{J-1} Q_{c_i,j}(t)+\sum_{j=1}^J V_{s_j(t),c_i}(t)\Big)$} every bounded set of states is uniformly small.
If complete resource pooling holds then  the fluid model of the process $\mathscr{Z}^0(t)$ for the unlimited supply of work system is stable.  Hence, by the fundamental theorem of Dai \cite{dai:95}, the process is ergodic, with some stationary measure $\nu^0$.
Returning to the original system, and the process $\mathscr{Z}(t)$, we have $\mathscr{Z}(t)  = \big( \mathscr{Z}^0(t),U(t),Q_{c_i,J}(t) \big) $
where the process $\mathscr{Z}(t)$ is transient because by $\rY_J(t) < \lambda t$ we have $Q_{c_i,J}(t)\to \infty$ as $t\to\infty$ almost surely.  However, the process $\mathscr{Z}(t)$ exactly satisfies the conditions of the  Lemma of Adan, Foss, Shneer and Weiss \cite{adan-foss-shneer-weiss:15}.  It follows that the distribution of $\mathscr{Z}^0(t)$ converges to $\nu^0$.
\end{proof}


We discuss the meaning of this theorem.  It says that under complete resource pooling if the arrival rate is high the queue in front of all the servers will grow linearly but the servers will stay close together and move at some joint average rate, so that the permutation of the servers and the queues of customers between them will tend to a stationary distribution.

We consider now the case that there is no complete resource pooling.  Consider a partition $(S_1,\ldots,S_L)$, let $S_\ell = \{M_k,\ldots,M_l\}$ be the servers in positions $k,\ldots,l$.  We denote by $Q_\ell = (Q_{c_i,k},\ldots,Q_{c_i,l-1},\,c_i=1,\ldots,I)$  the queues of customers between the servers of $S_\ell$, and by $v_\ell=(v_{M_k,c_i},\ldots,v_{M_l,c_i},\,i=1,\ldots,I)$ the remaining processing times of the servers of $S_\ell$.
\begin{theorem}
\label{thm.decomposedergodic}
Assume that for all $t>0$ there is a fixed partition $(S_1,\ldots,S_L)$ such that $\rY_{S_1}(t)<\cdots \rY_{S_L}(t)<\lambda t$ and $\drY_{S_1}(t)<\cdots \drY_{S_L}(t)<\lambda $.  Consider the processes $\mathscr{Z^\ell}(t) = \big(M_j(t),Q_\ell(t),v_\ell(t):\,\,M_j\in S_\ell\big)$.  Then there exist measures $\nu^\ell$ on the state spaces of $\mathscr{Z^\ell}(\cdot)$ such that  as $t\to \infty$ the distribution of   $\mathscr{Z^\ell}(t)$ converges to $\nu^\ell$ for $\ell=1,\ldots,L$.
\end{theorem}
\begin{proof}
Consider the subsystem of $M_j\in S_\ell$, and
$c_i\in \C(S_\ell)\backslash \C(S_{\ell+1}\cup\cdots\cup S_L)$.  With infinite supply of jobs of these types the system will be ergodic with stationary measure $\nu^\ell$.
The Theorem again follows from the Lemma of Adan, Foss and Weiss \cite{adan-foss-shneer-weiss:15}.
\end{proof}

When there is no resource pooling and the arrival rate is high, the servers will split to  subsets which move together at some joint average rate, tending to a stationary distribution of the queues inside each subset, but the queues separating these subsets of servers will grow without bound.

In general, fluid model equations (\ref{eqn.busyrate}), (\ref{eqn.gengetherfluidB}) do not determine the paths of $\rY$,  and do not provide us with a way of verifying  complete resource pooling or stability of the fluid model.
   This is not simply because we have not found the right fluid model equations necessary for that calculation.  The fact is that for general bipartite graphs with service rates $\mu_{s_j,c_i}$ that depend on both server and the customer type, under FCFS,
   first order and second order moment information alone does not determine the fluid limits of the system.
 This  was discovered in the seminal paper of Foss and Chernova \cite{foss-chernova:98}.  They
consider a system with 3 servers, 3 customer types and an almost complete bipartite compatibility graph as illustrated in Figure \ref{fig.3x3}.
\begin{figure}[htp]
\centering
\includegraphics[width=.25\textwidth]{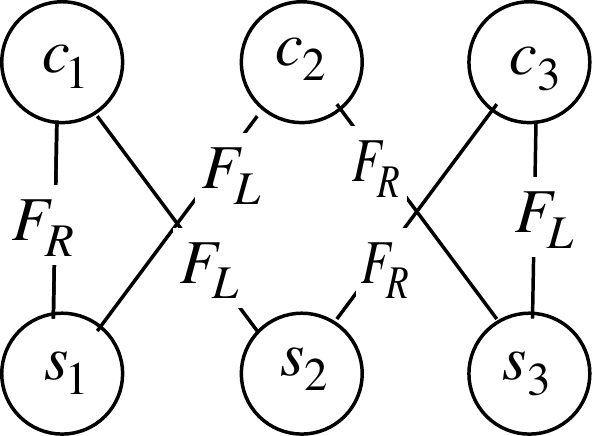}
\caption{A symmetric system with an almost complete 3 server 3 customer types graph.}\label{fig.3x3}
\end{figure}
Here $\alpha_{c_1}=\alpha_{c_2}=\alpha_{c_3}=1/3$, and the service time distributions
are $v_{s_1,c_1} \sim F_R,v_{s_1,c_2} \sim F_L,\,\,v_{s_2,c_1} \sim F_L,\,v_{s_2,c_3} \sim F_R,\,v_{s_3,c_2} \sim F_R,\,v_{s_3,c_3} \sim F_L$, with means $m_L\ne m_R$, so that each server has two service time distributions, and each customer type has two service time distributions.  Foss and Chernova show that for some fixed $\lambda,m_L,m_R$ it is possible to choose $F_L,F_R$ in such a way that the system under FCFS (they actually consider the equivalent JSW policy) is positive Harris recurrent, but under a different choice of $F_L,F_R$ it is transient.

In the rest of the paper we impose further assumptions on the service rates or on the shape of the bipartite graph, under which we derive more detailed fluid model equations.  With the aid of these we can  verify complete resource pooling and  stability of the fluid limits, and find conditions under which they hold.


\section{Service rates depend only on server}
\label{sec.serverdependent}
We now consider the special case where service rates depend only on the server (SD), and not on the customer type which he serves.  We let $m_{s_j}$ and $\mu_{s_j}=1/m_{s_j}$ be the mean service time and the service rate of server $s_j$.  Define $\mu = \sum_{s_j\in \S} \mu_{s_j}$ and $\beta_{s_j}=\mu_{s_j}/\mu$.
Then $\mu$ is the total service capacity of the system, and $\beta_{s_j}$ is the fraction of service capacity provided by server $s_j$.  For a subset $S$ of server types we use the notation $\mu_S=\sum_{s_j\in S} \mu_{s_j},\,\beta_S=\sum_{s_j\in S} \beta_{s_j}$.
In that case we have immediately:
\begin{corollary}
\label{thm.jointserverdeprates}
Assume $\mu_{s_j,c_i}=\mu_{s_j},\,c_i\in \C(s_j),\, j=1,\ldots,J$.    Under the conditions of Theorem {\rm\ref{thm.generaljointserversdynamics}}  a.s. all fluid limits at $s<\tau<t$ must satisfy:
\begin{equation}
\label{eqn.jointserverfluid1}
 \drY_j(\tau) =  \frac{\mu_{\{M_k,\ldots,M_l\}}}
 {\alpha_{\C(M_k,\ldots,M_l)\backslash \C(M_{l+1},\ldots,M_J)}},\quad j=k,\ldots,l.
\end{equation}
\end{corollary}
\begin{proof}
Substituting $\mu_{M_j,c_i} = \mu_{M_j}$ into (\ref{eqn.gengetherfluidB}), and summing over $c_i  \in \C(M_k,\ldots,M_l)\backslash \C(M_{l+1},\ldots,M_J)$ we obtain:
\[
\drY_l(\tau) \alpha_{\C(M_k,\ldots,M_l)\backslash \C(M_{l+1},\ldots,M_J)}   =
\sum_{j=k}^l \mu_{M_j}  \sum_{c_i\in\C(M_j)\backslash \C(M_{l+1},\ldots,M_J)}   \dot\rT_{M_j,c_i}(\tau)
\]
and using (\ref{eqn.busyrate}) the corollary follows.
\end{proof}

This shows that in the SD special case, indeed all the fluid trajectories of $\rY_j$ are along straight lines, as in Figure \ref{fig.fluiddynamics}.  The following theorems and definition characterize the fluid limits of $\rY_j$ completely.  The proofs of these theorems was given in Proposition B10 in \cite{adan-weiss:14}.  We present a slightly simplified proof here for completeness.

{\bf Condition for complete resource pooling in the SD case}:  For every subset of servers $S \ne \emptyset, \S$ and every subset of customer types  $C \ne \emptyset, \C$, the following 3 equivalent conditions hold:
\begin{equation}
\label{eqn.CRPSD}
\beta_{\S(C)} > \alpha_{C}, \qquad  \alpha_{\C(S)} > \beta_S, \qquad \beta_S>\alpha_{\U(S)}.
\end{equation}

The following Lemma has often been used in proofs of fluid stability (see \cite{dai-weiss:96}), and is useful here:
\begin{lemma}
\label{thm.simple}
Let $g(t)$ be an absolutely continuous nonnegative function on $t\ge 0$ and let $\dot{g}(t)$
denote its derivative whenever it exists.

{\rm(i)}
If $g(t)=0$ and $\dot{g}(t)$ exists, then $\dot{g}(t)=0$.

{\rm(ii)} Assume that for some $\epsilon >0$, whenever $g(t)>0$ and $\dot{g}(t)$ exists,  then  $\dot{g}(t)< - \epsilon$.   Then $g(t)=0$ for all $t>\delta$ where $\delta=g(0)/\epsilon$.  Furthermore $g(\cdot)$ is nonincreasing and hence, once it reaches zero, it stays there forever.
\end{lemma}

\begin{theorem}
\label{thm.pooledfluiddynamics}
{\rm(i)} Assume that condition {\rm(\ref{eqn.CRPSD})} holds, then complete resource pooling holds, that is, for any initial conditions there exists $t_0$ such that for every fluid limit $\rY_1(t)=\cdots=\rY_J(t)=\min(\mu t,\lambda t)$ holds for $t>t_0$.

{\rm(ii)} Assume that condition {{\rm(\ref{eqn.CRPSD})} only holds with $\ge$ instead of $>$}.  Then complete weak resource pooling holds.

{\rm(iii)} Assume that  complete resource pooling condition {\rm(\ref{eqn.CRPSD})} is strictly violated.   Then it is not possible to have   $\rY_1(\tau)=\cdots=\rY_J(\tau)<\lambda \tau$  for all $\tau$ in an interval $s<\tau<t$.
\end{theorem}
\begin{proof}
(i)  Assume that (\ref{eqn.CRPSD}) holds, and that at time $t$ the servers are split into the ordered partition $\bar{\S}(t)=(\rS_1,\ldots,\rS_L)$, and each of these subsets of servers are moving together.

By Corollary \ref{thm.jointserverdeprates},
\[
\drY_1(t) = \mu \frac{\beta_{\rS_1}}{\alpha_{\U(\rS_1)}},  \qquad
\drY_J(t) = \min\Big( \lambda, \mu \frac{\beta_{\rS_L}}{\alpha_{\C(\rS_L)}} \big).
\]
By (\ref{eqn.CRPSD}), $\frac{\beta_{\rS_1}}{\alpha_{\U(\rS_1)}}>1$ while
 $\frac{\beta_{\rS_L}}{\alpha_{\C(\rS_L)}} <1$.  Hence $\drY_1(t) > \mu$ while $\drY_J(t) \le\min(\lambda,\mu)$, so that
 $\frac{d}{dt} \big(\rY_J(t) -  \rY_1(t)\big) < 0$.
 By looking at the finite number of all different splits we can find $\epsilon>0$ such that
 $\frac{d}{dt} \big(\rY_J(t) -  \rY_1(t)\big) < \epsilon <0$.  {(i) then follows from  Lemma
\ref{thm.simple}.}

(ii)  Assume first that for some $S$, $\beta_S = \alpha_{\U(S)}$, in which case also
$\beta_{\overline{S}} = \alpha_{\C(\overline{S})}$, and consider the case that for all other subsets, (\ref{eqn.CRPSD}) holds.  Assume at time $t$ a partition $\bar{\S}(t)=(S_1,\ldots,S_L)$ in which $S_1$ is netiher $S$ nor $\overline{S}$.  In that case by the argument of (i), $\drY_1(t) > \drY_J(t)$.  This shows that for some $t_0$ we have for all $t>t_0$ the trajectories are given by the partition  $\bar{\S}(t)=\{S,\overline{S}\}$.   The proof for any number of weak inequalities in   (\ref{eqn.CRPSD}) follows by induction.

(iii)  If  resource pooling is strictly violated then there exists a subset of the servers, $S=\{M_1,\ldots,M_L\}$, such that $\beta_S < \alpha_{\U(S)}$.
Assume contrary to the statement of the proposition that there exists a fluid limit for which $\rY_1(\tau)=\cdots=\rY_J(\tau)<\lambda \tau$ for $\tau\in [t,t+\Delta],\, \Delta>0$.  Denote the common value of $\rY_j,\drY_j,\,j=1,\ldots,J$ by $\rY_{\mbox{Common}},\drY_{\mbox{Common}}$.
Consider customer types $c_i \in \U(M_1,\ldots,M_L)$.
By  (\ref{eqn.busyrate})-(\ref{eqn.gengetherfluidB}) we have:
\[
\drY_{\mbox{Common}}(\tau) \alpha_{\U(M_1,\ldots,M_L)} =
 \sum_{j=1}^L \mu_{M_j} \Big(\sum_{c_i\in \U(M_1,\ldots,M_L)} \drT_{M_j,c_i} (\tau)\Big) \le  \sum_{j=1}^L \mu_{M_j}
\]
Hence we obtain
\[
\drY_{\mbox{Common}}(\tau) \le \mu \frac{\beta_S}{\alpha_{\U(S)}} < \mu.
\]
On the other hand, if $\rY_1(\tau)=\cdots=\rY_J(\tau)<\lambda \tau$ for $\tau\in [t,t+\Delta],\, \Delta>0$, then by summing (\ref{eqn.gengetherfluidB}) over all servers and all customer types and using (\ref{eqn.busyrate}), we obtain $\drY_{\mbox{Common}}(\tau) = \mu$.  This contradiction proves (ii).
\end{proof}

\begin{definition}
\label{thm.subsetpooling}
Consider a partition of the servers into subsets $S',S,S''$.  We say that $S$ satisfies complete resource pooling condition {\rm(\ref{eqn.CRPSD})} between $S'$ and $S''$ (the order of $S'$ before $S''$ is important here), if the subsystem which consists of servers $s_j\in S$, and the customer types $c_i \in \C(S)\backslash \C(S'')$, with
$\tilde{\beta}_{s_j}=\beta_{s_j}/\beta_S$, $\tilde{\alpha}_{c_i} = \alpha_{c_1}/\alpha_{\C(S)\backslash \C(S'')}$, satisfies {\rm(\ref{eqn.CRPSD})}.
\end{definition}

We now have the following theorem, which enables us to trace the exact piecewise linear trajectories of the fluid model of the system:

\begin{theorem}
\label{thm.pooleddynamics}
Consider a fluid limit  with $\rY_{k-1}(t)<\rY_k(t)=\cdots=\rY_l(t)<\rY_{l+1}(t) \le \lambda t$ for some $t$   and  let $\bar{\S}(t)=(S',\{M_k,\ldots,M_l\},S'')$ be the corresponding partition of the servers.  Then
\begin{equation}
\label{eqn.jointserverfluid}
\drY_j(t) = \mu \frac{\beta_{\{M_k,\ldots,M_l\}}}{\alpha_{\C(M_k,\ldots,M_l)\backslash \C(S'')}},\quad j=k,\ldots,l
\end{equation}
during  $t<\tau<t+\Delta$ for some $\Delta>0$, if and only if $\{M_k,\ldots,M_l\}$ satisfies  complete resource pooling condition of Definition {\rm\ref{thm.subsetpooling}}  between $S'$ and $S''$.
\end{theorem}
\begin{proof}
If $\rY_{k-1}(t)<\rY_k(t)=\cdots=\rY_l(t)<\rY_{l+1}(t) \le \lambda t$ then by continuity, for some $\Delta$:
$\rY_{k-1}(\tau)<\rY_k(\tau)\le \cdots\le \rY_l(\tau)\le \rY_{l+1}(\tau) \le \lambda \tau$ for $t<\tau<t+\Delta$,  and so going back to the originating $\omega$ and subsequence $r$ for large enough $r$, we will have
 $\rY^r_{k-1}(r \tau)<\rY^r_k(r \tau)\le \cdots\le \rY^r_l(r \tau)\le \rY^r_{l+1}(r \tau) \le \lambda r \tau$ for $rt<r\tau<rt+r \Delta$.  In other words, servers $M_k,\ldots,M_l$ will serve customer types
  $c_i \in C(M_k,\ldots,M_l)\backslash C(M_{l+1},\ldots,M_J)$ as an isolated FCFS-ALIS sub-system, in the time interval $(r t, r t + r \Delta)$.  The theorem then follows by applying Theorem \ref{thm.pooledfluiddynamics} to this subsystem.
\end{proof}

\begin{corollary}
\label{thm.fluidstability}
Under complete resource pooling, the fluid model is stable if and only if $\lambda < \mu$
\end{corollary}

It is shown in \cite{adan-weiss:14} that if resource pooling does not hold then there exists a unique decomposition of the system into subsystems $(\S^{(1)},\C^{(1)}),\ldots,(\S^{(L)},\C^{(L)})$ with $\C^{(\ell)} = \U(\S_1\cup\cdots\cup\S^{(\ell)})\backslash \U(\S_1\cup\cdots\cup\S^{(\ell-1)})$ and  service rates $\mu^{(\ell)}= \mu \beta_{\S^{(\ell)}}$, and there are then values $\lambda^{(1)}<\cdots<\lambda^{(L)}$ so that system $(\S^{(\ell)},\C^{(\ell)})$ on its own is stable for all $\lambda < \lambda^{(\ell)}$, and the combined system exhibits local stability.  These results carry over to our system.

In summary, for the SD case we get the complete traces of the fluid model of the system, including answers to questions of stability, resource pooling, or decomposition, under FCFS policy.  In fact the fluid models are independent of the service time distributions, and depend only on first order moments.  In particular, the results are the same as those obtained for the system with Poisson arrivals and exponential service rates.

On the question of matching rates, the fluid model in not informative enough.  While we can obtain matching rates in the Poisson-exponential case as done in \cite{adan-weiss:11}, we cannot calculate matching rates for general service time distributions in the SD case.  We return to this question in Section \ref{sec.simulation}.  Matching rates can be calculated for some special bipartite graphs --- we do that in Section \ref{sec.computablematchingrates}.


\section{Service rates depend only on customer type}
\label{sec.customerdependent}

We now consider the special case where service rates depend only on the customer type (CD), regardless of which of the compatible servers is serving.  We let $m_{c_i}$ and $\mu_{c_i}=1/m_{c_i}$ be the mean service time and the service rate for customer type $c_i$.  In that case the total service capacity of the system is $|\S|=J$, for the $J$ servers, but capacity for each subset $C$ of customer types is $|\S(C)|$, the number of compatible servers.

In that case we have immediately:
\begin{corollary}
\label{thm.jointcustdeprates}
Assume $\mu_{s_j,c_i}=\mu_{c_i}$, for $\,s_j\in \S(c_i),\, i=1,\ldots,I$.    Under the conditions of Theorem \ref{thm.generaljointserversdynamics}  a.s. all fluid limits at $s<\tau<t$ must satisfy:
\begin{equation}
\label{eqn.jointserverfluidCD}
 \drY_j(\tau) =  \frac{l-k+1}
 {\sum_{c_i\in \C(M_k,\ldots,M_l)\backslash \C(M_{l+1},\ldots,M_J)}  \alpha_{c_i} m_{c_i}},\qquad j=k,\ldots,l.
\end{equation}
\end{corollary}
\begin{proof}
Substituting $\mu_{M_j,c_i} = \mu_{c_i}$ into (\ref{eqn.gengetherfluidB}), we have:
\[
 \drY_j(\tau)  \alpha_{c_i} m_{c_i}  = \sum_{j=k}^l  \dot\rT_{M_j,c_i}(\tau),
 \qquad c_i\in \C(M_k,\ldots,M_l)\backslash \C(M_{l+1},\ldots,M_J).
\]
and summing over all $c_i\in \C(M_k,\ldots,M_l)\backslash \C(M_{l+1},\ldots,M_J)$  we get by
(\ref{eqn.busyrate}) that
\[
 \drY_j(\tau) \sum_{c_i\in \C(M_k,\ldots,M_l)\backslash \C(M_{l+1},\ldots,M_J)}
  \alpha_{c_i} m_{c_i}  = l - k +1.
\]
\end{proof}

This shows that also in the special case of CD  all the fluid trajectories of $\rY_j$ are along straight lines, as in Figure \ref{fig.fluiddynamics}.
The following definition and theorems characterize the fluid limits of $\rY_j$ completely.

{\bf Condition for complete resource pooling in the CD case}:  For every subset of servers $C \ne \emptyset, \C$:
\begin{equation}
\label{eqn.CRPCD}
\frac{|\S(C)|}{|\S|} >
\frac{\sum_{c_i\in C} \alpha_{c_i} m_{c_i}}
{\sum_{c_i\in \C} \alpha_{c_i} m_{c_i}}.
\end{equation}

\begin{theorem}
\label{thm.pooledCD}
{\rm(i)} If condition {\rm(\ref{eqn.CRPCD})} holds then complete resource pooling holds, i.e., for some $t_0$ and for any $\lambda$, $\rY_1(t)=\cdots=\rY_J(t)$ for all $t>t_0$.

{\rm(ii)}  If {{\rm(\ref{eqn.CRPCD})} holds only with $\ge$ replacing $>$}, then complete weak resource pooling holds.

{\rm(iii)} If {\rm(\ref{eqn.CRPCD})}  is strictly violated then complete resource pooling does not hold.
\end{theorem}
\begin{proof}
If $\rY_1(\tau)=\cdots=\rY_J(\tau)$  for $s<\tau<t$ then by (\ref{eqn.jointserverfluidCD})
\begin{equation}
\label{eqn.CDspeed}
\drY_j(\tau) = \frac{|\S|}{\sum_{c_i\in \C} \alpha_{c_i} m_{c_i}}
\end{equation}

Assume that $\rY_J(t) > \rY_1(t)$ and assume partition $S_1(t),\ldots,S_L(t)$.
{We show that  if (\ref{eqn.CRPCD}) holds then there exists $\varepsilon>0$ such that $\drY_1(t)-\drY_J(t)\geq \varepsilon$, which by Lemma \ref{thm.simple} proves that complete resource pooling holds.

Indeed,   by Corollary \ref{thm.jointcustdeprates} and  (\ref{eqn.CRPCD}):}
\[
\drY_{S_1} =  \frac{|S_1|}{\sum_{c_i\in \U(S_1)} \alpha_{c_1} m_{c_i}} > \frac{|\S|}{\sum_{c_i\in \C} \alpha_{c_i} m_{c_i}},
\]
On the other hand:
\[
\frac{|\S|}{\sum_{c_i\in \C} \alpha_{c_1} m_{c_i}}  =
\frac{|\S\backslash S_L| + |S_L|}
{ \sum_{c_i\in \C \backslash \C(S_L)} \alpha_{c_i} m_{c_i}
+ \sum_{c_i\in \C(S_L)} \alpha_{c_i} m_{c_i}
}
\]
and $\C \backslash \C(S_L) = \C(\S\backslash S_L)$, and hence
$\frac{|\S\backslash S_L| }{ \sum_{c_i\in \C \backslash \C(S_L)} \alpha_{c_i} m_{c_i}} >
 \frac{|\S|}{\sum_{c_i\in \C} \alpha_{c_i} m_{c_i}}$ which implies that:
 \[
\drY_{S_L} =  \frac{|S_L|}{\sum_{c_i\in \C(S_L)} \alpha_{c_1} m_{c_i}} < \frac{|\S|}{\sum_{c_i\in \C} \alpha_{c_i} m_{c_i}},
\]

The proof of (ii) is similar to the proof of (ii) in Theorem \ref{thm.pooledfluiddynamics}.

If the condition (\ref{eqn.CRPCD})  is strictly violated then clearly it is not possible to have
$\rY_1(\tau)=\cdots=\rY_J(\tau)$  for $s<\tau<t$.  If it is only weakly violated, i.e., there exists $C,\S(C)$ such that $\frac{|\S(C)|}{|\S|} =
\frac{\sum_{c_i\in C} \alpha_{c_i} m_{c_i}}
{\sum_{c_i\in \C} \alpha_{c_i} m_{c_i}}$,  then if initially servers $\S(C)$ are behind all the others, they will never catch up with $\rY_J$.
\end{proof}

All the trajectories of the fluid limits for the CD case can be determined by the following Corollary, which mimics Theorem \ref{thm.pooleddynamics}, and has the same proof.
\begin{corollary}
\label{thm.pooleddynamics2}
Consider a fluid limit with $\rY_{k-1}(t)<\rY_k(t)=\cdots=\rY_l(t)<\rY_{l+1}(t) \le \lambda t$ for some $k, l$, and $t$, and let $\bar{\S}(t)=(S',\{M_k,\ldots,M_l\},S'')$ be the corresponding partition of the servers.  Then
\begin{equation}
\label{eqn.jointserverfluid3}
\drY_j(t) = \frac{ l - k +1 }
{ \sum_{c_i\in \C(M_k,\ldots,M_l)\backslash\C(S'')} \alpha_{c_i} m_{c_i} \Big/
\alpha_{\C(M_k,\ldots,M_l)\backslash\C(S'')}
}
\end{equation}
during $t<\tau<t+\Delta$ for some $\Delta>0$,  if and only if the subsystem consisting of
 $\{M_k,\ldots,M_l\}$ and $\C(M_k,\ldots,M_l)\backslash \C(S'')$ satisfies  condition {\rm(\ref{eqn.CRPCD})}.
\end{corollary}

As in the SD case, we get a complete picture of the fluid model in the CD case.  Similar to the SD case however, the fluid model does not contain enough information to calculate the matching rates.

\section{Systems with computable matching rates}
\label{sec.computablematchingrates}
For some types of compatibility graphs it is possible to calculate the matching rates of the fluid model, and in those cases one can again show that the fluid levels are piecewise linear, and calculate their trajectories.  We consider two such special types of networks:  networks with complete bipartite compatibility graph and networks with tree compatibility graph, as well as their hybrid.   The fluid models for these systems under FCFS were considered by Talreja and Whitt \cite{talreja-whitt:07}, for the SD case.  In our derivations here we allow service rates to depend both on the server and on the customer type.

\subsection{Network with complete bipartite compatibility graph}
\label{sec.completegraph}
We now assume that every server can serve all types of customers, i.e., the compatibility graph is a complete bipartite graph.   If all the servers can serve all the customers, then servers will never skip customers, and in effect the system will just behave like a GI/GI/$J$ queueing system with non-identical servers.

When a server will complete service he will immediately overtake all the other servers and will start serving the first waiting customer.   Average service time for server $s_j$, service rate for server $s_j$, and  total service capacity of the system, are then:
\begin{equation}
\label{eqn.completegraphrates}
\mu = \sum_{j=1}^J \mu_{s_j},  \qquad \mu_{s_j} = {m_{s_j}}^{-1}, \qquad
m_{s_j} = \sum_{c_i\in \C(s_j)} \alpha_{c_i} m_{s_j,c_i},
\end{equation}
and we can calculate matching rates as follows:
\begin{equation}
\label{eqn.completegraphmatching}
r_{s_j,c_i} = \frac{\mu_{s_j}}{\mu} \frac{\alpha_{c_i}}{\sum_{c_k \in \C(s_j)} \alpha_{c_k}}.
\end{equation}
Using the same arguments as for GI/GI/$J$ we get:
\begin{theorem}
For the case of a complete bipartite compatibility graph, under FCFS-ALIS policy, there is complete resource pooling always, and for every fluid model almost surely
\[
\rY_1(t)=\cdots = \rY_J(t) = \min \big(\rY_J(0) + \mu t, \lambda t\big), \quad t>0,
\]
where $\mu$ is given in (\ref{eqn.completegraphrates}). The matching rates while $\rY_J(t)< \lambda t$  are given by
$r_{s_j,c_i}$ in (\ref{eqn.completegraphmatching}).
\end{theorem}
\begin{proof}
Recall that $\rY_J(0) \le 0$ in our system description.  The matching rates
correspond to the fact that  server devotes a fraction $\alpha_{c_i}/\sum_{c_k\in \C(s_j)}\alpha_{c_k} $ of his services to type $c_i$.  The service rate of each server is given by (\ref{eqn.completegraphrates}) as long as there is a queue, and so his fraction of all services is  $\mu_{s_j}/\mu$, and (\ref{eqn.completegraphmatching}) follows.
\end{proof}

\subsection{Network with tree bipartite compatibility graph}
\label{sec.treegraph}
A tree graph is a connected graph with no loops.  With $K$ nodes it will have  exactly $K-1$ edges, and it will always have at least two leaves  (nodes that are connected by a single edge).   Furthermore, any sub-graph will be a tree, or a union of disconnected trees (a forest).
We now assume that the bipartite graph $\G$ is a tree.  It has $I+J$ nodes and therefore it has $I+J-1$ compatible pairs (edges), and at least two leaves, each of which can be either  a server or customer type.

Let $S_1,\ldots,S_L$ be an ordered partition of $s_1,\ldots,s_J$.  Denote by $C_\ell= \C(S_\ell)\backslash \C(S_{\ell+1}\cup\cdots\cup S_L)$.
Consider now a fluid  limit $\rY_1(\tau), \ldots, \rY_J(\tau)$, with permutation $\rM_1(\tau),\ldots,\rM_J(\tau)$, and assume that for all $s<\tau<t$ the following holds:
\begin{eqnarray}
&&     \rM_j(\tau),\rM_{j'}(\tau) \in  S_\ell \Longrightarrow  \rP_{M_j}(\tau)=  \rP_{M_{j'}}(\tau),
\nonumber \\
&& \rM_j(\tau) \in  S_\ell \ \mbox{ and } \rM_{j'}(\tau) \in S_{\ell+1}
 \Longrightarrow  \rP_{M_j}(\tau) < \rP_{M_{j'}}(\tau) \mbox{ or } \rP_{M_j}(\tau) = \rP_{M_{j'}}(\tau)
\label{eqn.partition}  \\
&&  \qquad  \mbox{ and the subgraphs of } (S_\ell,C_\ell),\;
(S_{\ell+1},C_{\ell+1})   \mbox{ are not connected}
\nonumber
\end{eqnarray}
We denote the common value of $\rP_{M_j}(\tau)$ for $M_j \in S_\ell$ by $\rY_{S_\ell}(\tau)$.
Clearly, by the continuity of the $\rY_j(\cdot)$, such a partition is defined for every $t$ and for some $s < t$.
This partition is a refinement of the partitions discussed in Sections \ref{sec.fluidlimits}, \ref{sec.serverdependent}, where we further divide subset of servers that move together, so that each such subset will have the  property that each $(S_\ell,C_\ell)$ is connected.

\begin{theorem}
\label{thm.tree}
Assume that the bipartite compatibility graph is a tree, and consider the partition $(S_1,\ldots,S_L)$ as in {\rm(\ref{eqn.partition})} valid for $s<\tau<t$.
Then:

{\rm(i)}  Equations  {\rm(\ref{eqn.busyrate})-(\ref{eqn.gengetherfluidB})}  have a unique solution, for every $S_\ell\in (S_1,\ldots,S_L)$, and hence $\drT_{M_j,c_i}(\tau), \drY_j(\tau)$ are constant for $s<\tau<t$.   As a result, almost surely,  the fluid limit has unique continuous piecewise linear trajectories.

{\rm(ii)}  Consider the set of equations
 \begin{eqnarray}
&&   \sum_{c_i\in\C(s_j)}   \eta_{s_j,c_i} = 1 \quad j=1,\ldots,J, \nonumber\\
&&   \sum_{s_j\in \S(c_i)} \frac{\mu_{s_j,c_i}}{\alpha_{c_i}}  \eta_{s_j,c_i}  =  \mu,    \quad i=1,\ldots,I
 \label{eqn.treepooling}
  \end{eqnarray}
with the $I+J-1$ unknowns $\eta_{s_j,c_i},\ (s_j, c_i) \in \G$, and an additional unknown $\mu$.  The system will have complete resource pooling if and only if
{\rm(\ref{eqn.treepooling})}  has a positive solution, and it will have complete weak resource pooling if the solution is non-negative.

{\rm(iii)} If complete resource pooling holds then $\mu$ is the pooled service rate, and the matching rates are given by:
\begin{equation}
\label{eqn.treematching}
r_{s_j,c_i} = \frac{\mu_{s_j,c_i}  \eta_{s_j,c_i}}{\mu}.
\end{equation}
\end{theorem}
\begin{proof}
(i)    The equations  (\ref{eqn.busyrate})-(\ref{eqn.gengetherfluidB}) for each $S_\ell$ are:
\[
 \sum_{c_i\in C_\ell}   \dot\rT_{M_j,c_i}(\tau) = 1, \quad M_j \in S_\ell, \\
\]
\[
 \drY_{S_\ell}  =   \sum_{s_j\in S_\ell} \frac{\mu_{M_j,c_i}}{\alpha_{c_i}}  \dot\rT_{M_j,c_i}(\tau),
 \qquad c_i \in C_\ell.
\]
with unknowns $\drY_{S_\ell}$ and $\dot\rT_{M_j,c_i}(\tau)$ for each edge in the subgraph $(S_\ell,C_\ell)$.
Since the subgraph is a connected tree, the number of  unknowns
 is equal to the number of equations.  The equations are independent, so the solution is unique, for all $t < \tau< s$.   The solution must be non-negative, because the fluid limits exist.
 This proves that the fluid limit $\rY_{S_\ell}$ moves along a linear trajectory in the interval $(s,t)$.

We note that the equations can be solved in $|S_\ell|+|C_\ell | $ steps:  Locate a leaf in the graph. If it is $M_j$, it has a single customer type $c_i = \C(M_j) \cap C_\ell$, and $\dot\rT_{M_j,c_i}(\tau)=1$.  If it is $c_i$ it has a single server $M_j = \S(c_i)\cap S_\ell$, and then $\dot\rT_{M_j,c_i} = \drY_{S_\ell} \alpha_{c_i} m_{M_j,c_i}$.  In either case one can eliminate the leaf node and one equation and continue to solve for the remaining graph.  Note that deleting a leaf from a tree leaves a connected tree.

(ii)  Clearly if there is no positive solution to (\ref{eqn.treepooling}) then there can be no complete resource pooling (i.e., it is impossible to have $\rY_1(\tau)=\cdots =\rY_J(\tau)$ for any interval of $\tau$'s).  If the solution is non-negative with some 0 values for some edges, this implies that some disconnected subtrees   move at the same rate, but may have different initial positions.  So the system has complete weak resource pooling.
Assume that (\ref{eqn.treepooling}) has a positive solution.   We need to show that for some $t_0$ all the trajectories $\rY_j(t)$ meet for $t > t_0$.  Assume that at time $t$,  $\rY_1(t) < \rY_J(t)$.  We will show that $\drY_J(t) - \drY_1(t) < 0$, which by Lemma \ref{thm.simple} will complete the proof.

By continuity we have for an interval $t-\delta<\tau<t+\delta$ in which  the partition is $S=(S_1,\ldots,S_L)$, where  $S_1=(M_1,\ldots,M_k)$ and $S_L=(M_l,\ldots,M_J)$ and $\rY_{S_1}(\tau)<\rY_{S_L}(\tau)$.  We will show that $\drY_{S_1}(\tau) > \drY_{S_L}(\tau)$.

Denote by $\mu^{(S_1)},\eta_{M_j,c_i}^{(S_1)}$ and $\mu^{(S_L)},\eta_{M_j,c_i}^{(S_L)}$ the non-negative solutions of  (\ref{eqn.busyrate})-(\ref{eqn.gengetherfluidB}) for $S_1$ and for $S_L$, and by $\mu^{(0)},\eta_{s_j,c_i}^{(0)}$ the positive solution of (\ref{eqn.treepooling}).

We note that the solution of (\ref{eqn.busyrate})-(\ref{eqn.gengetherfluidB}) for the tree graphs  $(S_\ell,C_\ell),\,\ell=1,\ldots,L$, as well as for the complete tree graph $(\S,\C)$ are in fact the unique optimal solutions of the  corresponding linear programs (LP):
{
 \begin{eqnarray*}
 \label{eqn.subgraphLP}
&&\max \ \mu  \nonumber \\
&& s.t. \left\{
\begin{array}{lll}
   \sum_{c_i\in C_\ell}   \eta_{M_j,c_i} \le 1, &M_j\in S_\ell, \\
\sum_{M_j\in S_\ell}  \mu_{M_j,c_i}  \eta_{M_j,c_i}  \ge \mu \alpha_{c_i},
& c_i\in C_\ell,  \\
  \qquad \eta_{M_j,c_i} \ge 0. &  \
  \end{array}
  \right.
  \end{eqnarray*}
}
The fact that they are unique optimal solutions is explained in the following Section \ref{sec.optimal}.

Consider then the LP (\ref{eqn.subgraphLP}) for $(S_1,C_1)$, and substitute the values of $\mu^{(0)},\eta_{M_j,c_i}^{(0)}$.  We then have that:
\[
\sum_{M_j\in \S_1} \mu_{M_j,c_i}  \eta_{M_j,c_i}^{(0)}  = \mu^{(0)} \alpha_{c_i},
\qquad c_i\in C_1,
\]
because $\S(C_1) = S_1$, since $C_1$ includes customers that were skipped by all the other servers.   At the same time:
\[
\sum_{c_i\in C_\ell)}   \eta_{M_j,c_i}^{(0)} <1, \qquad \mbox{for at least one } M_j\in S_1,
\]
because  the  graph of $(\S,\C)$ is connected, and therefore there exists a link from some server  $M_j\in S_1$ to a customer type $c_i \not\in C_1$, and by assumption $\eta_{M_j,c_i}^{(0)} > 0$.
Hence this is a feasible but not optimal solution, which proves that $\mu^{(S_1)}>\mu^{(0)}$.

On the other hand, consider the LP (\ref{eqn.subgraphLP}) for $(\S_L,\C_L)$.  Because $C_L=\C\cap \C(S_L)$, it has all  the  constraints as the LP for $(\S,\C)$, with the additional constraints that $\eta_{M_j,c_i}=0$ whenever $M_j \not\in S_L$.  Hence the LP for $(S_L,C_L)$ is more constrained than that for $\S,\C$, and further more, in the optimal solution of $\S,\C$ all the $\eta_{M_j,c_i}^{(0)}>0$.  This implies that $\mu^{(S_L)} < \mu^{(0)}$.

But, $\mu^{(S_1)} = \drY_{S_1}=\drY_1,\;\mu^{(S_L)} = \drY_{S_L}=\drY_J$, and we have shown that if $\rY_1(t)< \rY_J(t)$ then $\drY_1(t)-\drY_J(t)>0$, as required.

(iii)  In the optimal solution of (\ref{eqn.treepooling}) the values of $\eta_{s_j,c_i}$ are the fractions of time allocated by server $s_j$ to customers of type $c_i$, and therefore the rate at which customers of type $c_i$ are processed by server $s_j$ is $\mu_{s_j,c_i} \eta_{s_j,c_i}$.
The total processing rate is then the sum of all these $\mu = \sum_{s_j,c_i\in \G} \mu_{s_j,c_i} \eta_{s_j,c_i}$, which is indeed the solution of (\ref{eqn.treepooling}).   The matching rates are therefore given by (\ref{eqn.treematching}).
\end{proof}

 {\begin{remark}
 A system is a hybrid of the systems  studied in Sections {\rm\ref{sec.completegraph}-\ref{sec.treegraph}}, if its bipartite compatibility graph consists of several complete graphs which are connected by a tree graph.
 For these hybrid systems one can again calculate the matching rates, and obtain a complete description of the fluid model trajectories.
\end{remark}
}


\section{Maximal throughput under FCFS}
\label{sec.optimal}
We consider a static planning problem similar to  Harrison and Lopez \cite{harrison-lopez:99}:
\begin{eqnarray}
&&\max \mu \nonumber \\
&&s.t. \left\{
\begin{array}{lll}
\sum_{c_i\in \C(s_j)} \eta_{s_j,c_i} \le 1, & s_j\in \S,  \\
 \sum_{s_j\in \S(c_i)} \mu_{s_j,c_i}  \eta_{s_j,c_i}  \ge  \alpha_{c_i} \mu,  & c_i \in \C,\\
 \eta_{s_j,c_i} \ge 0, & (s_j,c_i)\in\G
 \end{array}
 \right. \label{eqn.LP}
\end{eqnarray}
with the decision variables $\eta_{s_j,c_i}$, $(s_j,c_i)\in\G$ and $\mu$.
  Here $\eta_{s_j,c_i}$ is the fraction of time that server $s_j$ allocates to customers of type $c_i$, and $\mu$ is the rate at which the total stream of arrivals is served.
The first $J$ constraints (the server constraints)  say that the sum of  allocations for each server cannot exceed 1.
The next $I$ constraints (the customer constraints) say that the allocations $\eta_{s_j,c_i}$, are sufficient to serve the fraction customers of type $c_i$, to keep up with the total service rate $\mu$.
In terms of our system, $\eta_{s_j,c_i}$ can be thought of as long term average of $\drT_{s_j,c_i}$, and $\mu$ can be thought of as long term average of $\drY_1$, the rate of progress of $\rY_1$.

 The following Theorem is a simple consequence of Theorem 1 in the paper of Dai and Lin \cite{dai-lin:05}
 \begin{theorem}[Dai-Lin \cite{dai-lin:05}]
 \label{thm.maxthroughput}
 Let $\mu^*$ be the optimal value of the LP {\rm(\ref{eqn.LP})}.  Then under any policy, the fluid model is unstable if $\lambda>\mu^*$, so any policy that achieves fluid stability for all $\lambda<\mu^*$ is throughput optimal.
 \end{theorem}
\begin{proof}
Consider the departure processes of customers of type $c_i$.  Denote its fluid limits by $\rD_{c_i}(t)$, and let the fluid allocation rates be $ \drT_{s_j,c_i}(t)$.  Under any policy,  $\sum_{c_i\in \C(s_j)} \drT_{s_j,c_i}(t) \le 1$ needs to hold for all $t>0$ for all servers.  Also, for any fluid limit, under any policy $\drD_{c_i}(t) = \sum \mu_{s_j,c_i} \drT_{s_j,c_i}(t)$.
 It follows that the fluid model can only be stable if for every $c_i$,   $\drD_{c_i}(t) = \sum \mu_{s_j,c_i} \drT_{s_j,c_i}(t) \ge \lambda \alpha_{c_i}$.  Hence $\mu^*$ is an upper bound on $\lambda$ for which the fluid model can be stable.
   \end{proof}
 We note that  (\ref{eqn.LP}) is an optimization problem for a network with gains (cf. \cite{ahuja-magnanti-orlin:93}).  We now proceed to discuss the solution of  the problem (\ref{eqn.LP}) through a number of observations.
\begin{itemize}
\item[(i)]
The problem is feasible, since 0 for all decision variables is a solution.
\item[(ii)]
The problem is bounded, since $\mu$ is bounded by a positive linear combination of the $\eta_{s_j,c_i}$, and each $\eta_{s_j,c_i} \le 1$.
\item[(iii)]
The optimal value is $\mu^*>0$, since the problem is feasible if we take $\eta_{s_j,c_i}=\frac{1}{I\, J}$.
\item[(iv)]
The server constraints are satisfied as equalities in the optimal solution, since $\mu$ can only increase with every $\eta_{s_j,c_i}$.
\end{itemize}
We rewrite the LP and its dual, DP, in a slightly different form, including slack variables:
\begin{eqnarray*}
 {\rm LP} && \left\{
\begin{array}{ll}
&\max  \mu  \\
&  s.t.\left\{
\begin{array}{lll}
\sum_{c_i\in \C(s_j)} \eta_{s_j,c_i}= 1, & s_j\in \S,  \\
\mu - \sum_{s_j\in \S(c_i)} \frac{\mu_{s_j,c_i}}{\alpha_{c_i}} \eta_{s_j,c_i} + \theta_{c_i} =0, & c_i \in \C, \\
\eta_{s_j,c_i},\theta_{c_i}  \ge 0, & (s_j,c_i) \in {\cal G}.
\end{array}
\right.
\end{array}
\right. \\
{\rm DP} &&\left\{
 \begin{array}{ll}
&\min  \sum_{s_j\in \S}  y_{s_j}   \nonumber \\
&s.t.\left\{
\begin{array}{lll}
  \sum_{c_i\in \C}  z_{c_i} = 1, & \\
y_{s_j}  - \frac{\mu_{s_j,c_i}} {\alpha_{c_i}} z_{c_i} - x_{s_j,c_i} =  0, & (s_j,c_i) \in {\cal G}, \\
z_{c_i}, \;\,  x_{s_j,c_i}  \ge 0, & (s_j,c_i) \in {\cal G}.
\end{array}
\right.
\end{array}
\right.
\end{eqnarray*}
We observe that:
\begin{itemize}
\item[(v)]
In the optimal solution there is at least one  $\eta_{s_j,c_i}>0$ for each server $s_j$, and at least one  $\eta_{s_j,c_i}>0$ for each customer type $c_i$, since the server constraints are satisfied as equalities, and in the customer constraints $\mu>0$.
\item[(vi)]
Every basic optimal solution has no less than $\min\{I,J\}$ and no more than $I+J-1$ positive $\eta_{s_j,c_i}$, by (v) and since there are $I+J$ constraints and $\mu>0$.
\item[(vii)]
Since the primal is feasible and bounded, both the primal and the dual possess optimal solutions.
\item[(viii)]
In an optimal solution $y_{s_j} \ge 0$, since it needs to be $\ge$ than non-negative quantities.
\end{itemize}

The most important property of the solutions is the following results, which must be known and hidden in the literature on network flows with gains, but we could not find a good explicit reference and we provide a proof here.
\begin{lemma}
\label{thm.basictree}
The positive arcs in a basic solution of the LP {\rm(\ref{eqn.LP})} cannot contain a cycle.
\end{lemma}
\begin{proof}
Assume that a basic solution of LP (\ref{eqn.LP})  contains the columns of a cycle of arcs  of $\G$, which for simplicity we assume are labeled as $(s_1,c_1),(s_1,c_2),(s_2,c_2),(s_2,c_3),\ldots,(s_{L-1},c_L),(s_L,c_L),(s_L,c_1)$.    We will get a contradiction.
Denote $a_{j,i}=\frac{\mu_{s_j,c_i}} {\alpha_{c_i}}$, for these arcs.
Consider the complementary slack dual solution.  It will have $x_{s_i,c_i}=0$ for all the $2L$ arcs in the cycle.
That implies that $y_{s_j} = a_{j,i} z_{c_i}$ for all these arcs.  This implies that
$\frac{ a_{1,1 } a_{2,2} \cdots a_{L,L} } { a_{1,2 } a_{2,2} \cdots a_{L,1} } = 1$.  Consider now the square matrix formed by $2L$ columns corresponding to these arcs, and the $2L$ non-zero rows of these columns.  Its determinant is $a_{1,1 } a_{2,2} \cdots a_{L,L} - a_{1,2 } a_{2,2} \cdots a_{L,1} =0$.  Hence this cannot be a basis.
\end{proof}
In an optimal solution of the LP  (\ref{eqn.LP}) we refer to the arcs with positive values of $\eta_{s_j,c_i}$ as the solution graph.

\begin{theorem}
\label{thm.optimal}
Consider an optimal basic solution of the LP {\rm(\ref{eqn.LP})}.

{\rm(i)} Assume all the slacks $\theta_{c_i}=0$, the optimal solution graph is a tree, and all the $I+J-1$ basic variables $\eta_{s_j,c_i}>0$.  Then  erasing all the non-basic arcs and using FCFS for the remaining graph will achieve complete resource pooling and be throughput optimal with processing rate $\mu^*$.

{\rm(ii)} Assume all the slacks $\theta_{c_i}=0$, the optimal solution graph is a tree, but some of the $I+J-1$ basic variables are $=0$.  Then  erasing all the non-basic arcs, and all the basic arcs with $\eta_{s_j,c_i} = 0$, and using FCFS for the remaining graph will achieve complete weak resource pooling and be throughput optimal with processing rate $\mu^*$.

{\rm (iii)}  If for some $c_i$, $\theta_{c_i}>0$, let $C_1=\{c_i: \theta_{c_i}=0\}$ and let $S_1=\S(C_1)$, and assume that the subgraph $S_1,C_1$ is connected.   Then  in the solution graph $S_1,C_1$ are not connected to the remaining nodes.  Furthermore:  formulate the LP (\ref{eqn.LP}) for  the subsystem $S_1,C_1$ with the corresponding arcs of $\G$.  Then for this smaller problem either {\rm (i)} or {\rm(ii)} holds, and under FCFS complete resource pooling holds, the processing rate is $\mu_1=\mu^*$ and this policy is throughput optimal for $C_1$.

{\rm (iv)}  In the case of {\rm(iii)}, formulating the  LP {\rm(\ref{eqn.LP})} for the subgraph of $\S\backslash S_1,\C\backslash C_1$, the optimal solution will  have $\mu_2>\mu_1$.  Continuing in this way one obtains a unique decomposition of the the system to subgraphs $(S_1,C_1),\ldots,(S_L,C_L)$ each of which has an optimal tree solution, such that under FCFS it will have complete resource pooling, moving at rates $\mu_L>\cdots>\mu_1$, and these rates are maximal throughput for $(S_\ell,C_\ell)$ conditional on retaining the solutions of $(S_1,C_1),\ldots,(S_{\ell-1},C_{\ell-1})$.
\end{theorem}
\begin{proof}
(i) If in an optimal basic solution all the slacks $\theta_{c_i}=0$ then the solution will have $\mu>0$ and $I+J-1$ basic variables $\eta_{s_j,c_i}$ which by Lemma \ref{thm.basictree} have no loops and hence the solution graph is a tree.  We assume that all the arcs in the tree have $\eta_{s_j,c_i}>0$.   If we use only the arcs of the tree, we have a system with a tree bipartite graph, and by Section \ref{sec.treegraph}, this system under FCFS will have complete resource pooling and processing capacity $\mu^*$.   By Theorem \ref{thm.maxthroughput} this will be throughput optimal.

(ii) If some of the arcs of the solution tree have $\eta_{s_j,c_i}=0$, then by Theorem \ref{thm.tree} the system with only the arcs of the solution graph under FCFS will have complete weak resource pooling, with processing capacity $\mu^*$.   By Theorem \ref{thm.maxthroughput} this will be throughput optimal

(iii) Consider  $c_i\in C_1$, and assume that for some $s_j \in \S(c_i)$ and $c_k \not\in C_1$, the optimal solution has $\eta_{s_j,c_k} >0$.   it is then possible to reduce $\eta_{s_j,c_k} >0$ and increase $\eta_{s_j,c_i}$ for all $c_i\in C_1$, without violating the feasibility of the solution.  But this modified solution can only increase the objective value.  This proves that in the solution graph $S_1,C_1$ is not connected to any other parts of the system.  Hence, solving the reduced problem for $S_1,C_1$ the optimal solution will have $\theta_{c_i}=0,\,c_i\in C_1$, the solution graph will be a tree, and the optimal value for the reduced problem will be $\mu_1=\mu^*$.

(iv)  Clearly if for some $c_i$, $\theta_{c_i}>0$ then there must exist $(C_1,S_1)$ with a connected subgraph such that the conditions of (c) hold and $\mu_1$  equal to the optimal $\mu^*$ the value for the whole system.  This then is maximum throughput for $C_1$.  If we remove this $C_1$ and its servers, we can continue to decompose the remaining graph.
\end{proof}

The results of Theorem \ref{thm.optimal} are for a particular basic solution.  If there are several basic solutions, one might ask whether when using all the arcs of a non-basic solution and FCFS policy, there will be complete resource pooling and maximum throughput.  We do not currently know the answer in general.  For the special case of customer dependent service (CD) the answer is positive:  As shown in Section \ref{sec.customerdependent}, using the full bipartite graph we get complete resource pooling and maximum throughput under condition (\ref{eqn.CRPCD}).


\section{Exploration of the server dependent case under general service distributions}
\label{sec.simulation}

We have shown in Section \ref{sec.serverdependent} that the fluid model for the server-dependent case with general renewal arrivals and service times is the same as the fluid model for the Poisson exponential case.  In particular, the necessary and sufficient condition for complete resource pooling is insensitive to the service time distribution.  However, our fluid model analysis does not provide enough information to calculate the matching rates $r_{ij}$ for this more general case.  It is tempting to conjecture that if the system is overloaded, the matching rates will be the same as for the Poisson exponential case, and thus will be given by the matching rates calculated for the FCFS infinite bipartite matching model of \cite{adan-weiss:11}.

This, however, is not the case.  A simulation study reveals that in general, the matching rates in an overloaded, server-dependent system are sensitive to the service time distribution.  The matching rates of such a system under non-exponential service time distribution are very close to those under the exponential distribution, yet they are different, in a statistically significant manner.

We considered three topologies for the system, labeled 1--3, shown in Figure~\ref{fig.topologies}.  The topologies were parameterized by the customer type probabilities $\alpha$ and service rates $\mu$ as shown in Table~\ref{tab.simulation.parameters}.  For each topology, we used three service time distributions: Pareto (denoted by the subscript `p'), and two versions of the uniform distribution (`$\mathrm{u}_1$' and `$\mathrm{u}_2$').  In the Pareto case, we used a distribution having the density $f(x) = 3\gamma (\gamma x + 1)^{-4},\  x \ge 0$.  This Pareto distribution has only first and second finite moments, and is parameterized by a scale parameter $\gamma$, so that its mean is $1/2 \gamma$.  Thus, to achieve a service rate $\mu_{j}$, we set $\gamma = \mu_{j}/2$.  The two uniform distributions are $U(0, 2/\mu_{j}$) and $U(.9/\mu_{j}, 1.1/\mu_{j})$.

\begin{figure}[htp]
\centering
\includegraphics[width=.6\textwidth]{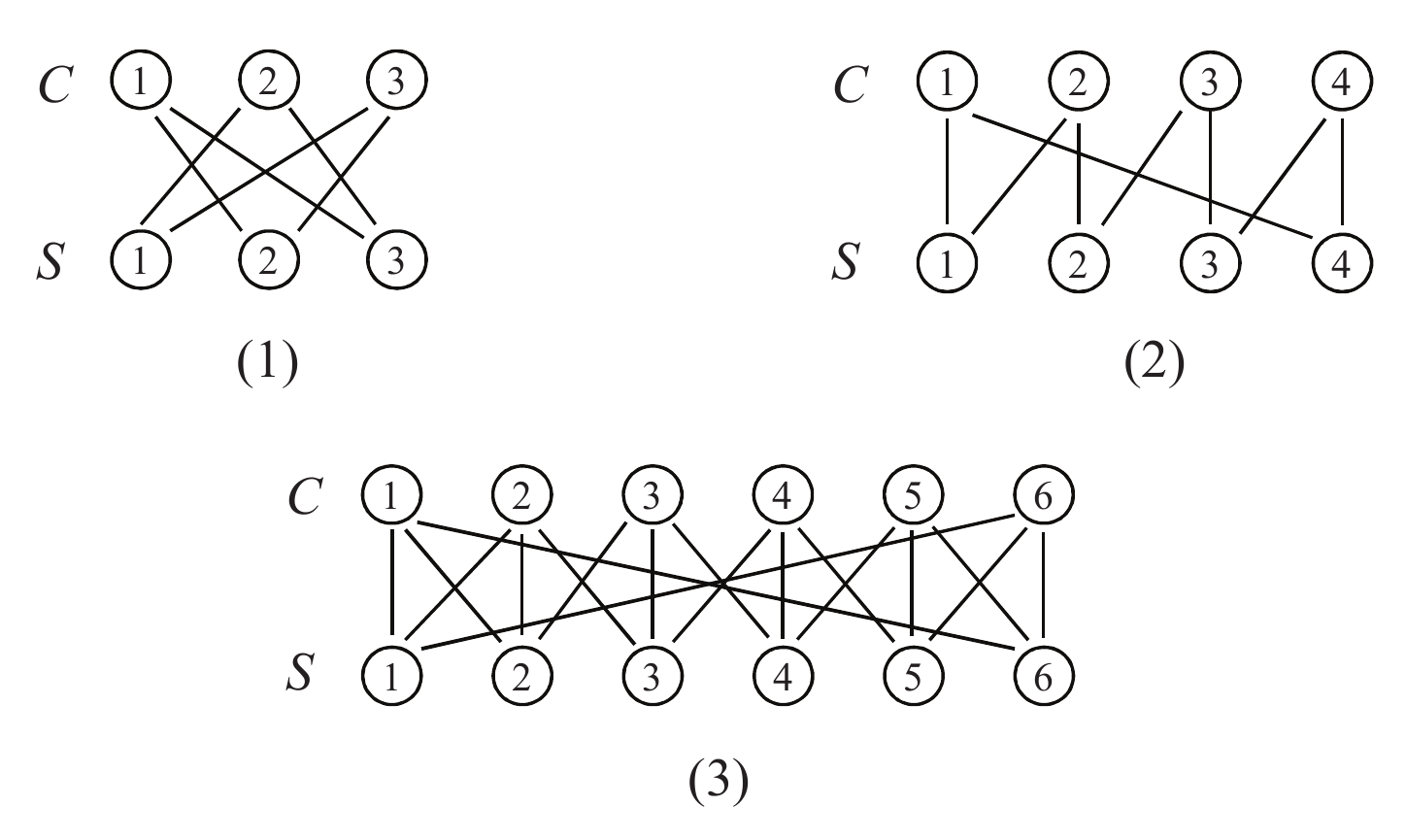}
\caption{Topologies of the systems for the simulation study.}\label{fig.topologies}
\end{figure}

\begin{table}[h]
\centering
\begin{tabular}{*{7}{c}}
\toprule
system 	& $\alpha$ 							& $\mu$ 									& exponential 	& Pareto 	& Uniform~1 & Uniform~2 \\
\midrule
1			& $(.2, .6, .2)$					& $(.4, .2, .4)$						& .285			& .299		& .535		& .074	\\
2			& $(.1, .4, .4, .1)$				& $(.4, .3, .2, .1)$					& .528			& *			& .0078		& *	\\
3			& $(.1, .2, .2, .1, .2, .2)$	& $(.05, .1, .15, .2, .2, .3)$	& .636			& *			& *			& *	\\
\bottomrule
\end{tabular}
\caption{System parameters for the simulation study, and resulting $p$-values of the Hotelling's $T^2$ test.  Asterisks denote $p$-value $< 10^{-15}$.}
\label{tab.simulation.parameters}
 \end{table}

In each simulation replication, the system was initialized with all servers simultaneously starting service of successive customers, each customer being randomly chosen from the server's compatibility set. To let the system approach steady state, it was first run for 100,000 service completions as a warmup period.  After warmup, the system was run for additional 1,000,000 service completions, and the fraction of services of customer type $c_i$ by server $s_j$ was recorded.  This procedure was repeated 100 times, and each element of the final estimated matrix $\widehat{r}$ is therefore a mean of 100 simulated fractions.  The simulation was carried out using the R programming language (www.r-project.org).

For each model, the matrix $r$ was analytically computed as described in \cite{adan-weiss:11}, and the estimated matrices $\widehat{r}_\mathrm{p}$, $\widehat{r}_\mathrm{u_1}$, and $\widehat{r}_\mathrm{u_2}$ were computed by simulation.  The resulting matrices are shown below.  The entries of the estimated matrices are invariably very close to the theoretical $r$ values, yet when comparing them using Hotelling's $T^2$ test (see below), it turns out in most cases that they are different in a statistically significant manner; see the last three columns of Table~\ref{tab.simulation.parameters}.  Interestingly, the matching rates in system 1 appear to be insensitive to the service time distribution.  As a control for the veracity of our simulation, we simulated the system also under exponential service time distribution, and as expected, did not get any significant test results; see column~4 of Table~\ref{tab.simulation.parameters}.

Hotelling's $T^2$ test is the multivariate generalization of the ubiquitous Student's $t$ test.  In each simulation replication, the non-zero entries of the empirical matching rate matrix $\widehat{r}$ (those corresponding to service compatibility) may be thought of as a realization of a random vector.  The entries of this vector, however, are dependent, as they must sum to~1.  The null hypothesis of the test is that the mean of this vector is the corresponding vector derived from the theoretical matching rate matrix $r$.  The test's statistics $T^2$ is a scaled sum of the squared deviations of the observed vectors from the hypothesized mean vector; under the null hypothesis, it possesses asymptotically an $F$ distribution.  To make the empirical covariance matrix of the observed (simulated) vectors invertible, the last entry of each vector was omitted.  For more details on Hotelling's $T^2$ test, see~\cite{krzanowski:00}.

\medskip

\noindent{\large \bf System 1}

\begin{align*}
r &=
\begin{pmatrix}
0 	  & .1 & .1 \\
.3  & 0   & .3 \\
.1  & .1 & 0
\end{pmatrix}, &
\widehat{r}_\mathrm{p} &=
\begin{pmatrix}
0 & .09996 & .10006 \\
.29987 & 0 & .30013 \\
.1 & .09997 & 0
\end{pmatrix} \\
\\
\widehat{r}_\mathrm{u_1} &=
\begin{pmatrix}
0 & .1 & .09996 \\
.30002 & 0 & .30003 \\
.10001 & .09998 & 0
\end{pmatrix}, &
\widehat{r}_\mathrm{u_2} &=
\begin{pmatrix}
0 & .10002 & .10009 \\
.29999 & 0 & .29991 \\
.10002 & .09998 & 0
\end{pmatrix} \\
\end{align*}

\medskip
\noindent{\large \bf System 2}

\begin{align*}
r &=
\begin{pmatrix}
.06443 & 0 & 0 & .03557 \\
.3356 & .06443 & 0 & 0 \\
0 & .2356 & .1644 & 0 \\
0 & 0 & .03557 & .06443
\end{pmatrix}, &
\widehat{r}_\mathrm{p} &=
\begin{pmatrix}
.06477 & 0 & 0 & .03524 \\
.33519 & .0648 & 0 & 0 \\
0 & .23523 & .16478 & 0 \\
0 & 0 & .03531 & .06468
\end{pmatrix} \\
\\
\widehat{r}_\mathrm{u_1} &=
\begin{pmatrix}
.06447 & 0 & 0 & .03553 \\
.33549 & .06447 & 0 & 0 \\
0 & .23556 & .16447 & 0 \\
0 & 0 & .03554 & .06446
\end{pmatrix} , &
\widehat{r}_\mathrm{u_2} &=
\begin{pmatrix}
.06465 & 0 & 0 & .03537 \\
.33535 & .06461 & 0 & 0 \\
0 & .2354 & .16465 & 0 \\
0 & 0 & .03535 & .06463
\end{pmatrix}
\end{align*}

\medskip
\noindent{\large \bf System 3}

\begin{align*}
r &=
\begin{pmatrix}
.004584 & .009497 & 0 & 0 & 0 & .08592 \\
.03825 & .06356 & .09819 & 0 & 0 & 0 \\
0 & .02694 & .04084 & .1322 & 0 & 0 \\
0 & 0 & .01097 & .02815 & .06087 & 0 \\
0 & 0 & 0 & .03963 & .06184 & .09854 \\
.007164 & 0 & 0 & 0 & .07729 & .1155
\end{pmatrix} \\
\\
\widehat{r}_\mathrm{p} &=
\begin{pmatrix}
.00462 & .01039 & 0 & 0 & 0 & .08499 \\
.03853 & .06318 & .0983 & 0 & 0 & 0 \\
0 & .02638 & .04062 & .13298 & 0 & 0 \\
0 & 0 & .01098 & .02787 & .0612 & 0 \\
0 & 0 & 0 & .03923 & .06149 & .09927 \\
.00681 & 0 & 0 & 0 & .07741 & .11575
\end{pmatrix} \\
\\
\widehat{r}_\mathrm{u_1} &=
\begin{pmatrix}
.00465 & .00894 & 0 & 0 & 0 & .08645 \\
.03781 & .06386 & .09834 & 0 & 0 & 0 \\
0 & .02718 & .04078 & .132 & 0 & 0 \\
0 & 0 & .0109 & .02819 & .06092 & 0 \\
0 & 0 & 0 & .03979 & .06198 & .09823 \\
.00754 & 0 & 0 & 0 & .07713 & .11531
\end{pmatrix} \\
\\
\widehat{r}_\mathrm{u_2} &=
\begin{pmatrix}
.00476 & .00864 & 0 & 0 & 0 & .08661 \\
.0374 & .06409 & .09857 & 0 & 0 & 0 \\
0 & .02727 & .04055 & .13213 & 0 & 0 \\
0 & 0 & .01088 & .02802 & .0611 & 0 \\
0 & 0 & 0 & .03985 & .06206 & .09811 \\
.00784 & 0 & 0 & 0 & .07684 & .11529
\end{pmatrix}
\end{align*}

A similar phenomenon occurs with the steady-state distribution of the \emph{server span} $Y_J(t) - Y_1(t)$, which is the distance between the leftmost and rightmost servers along the stream of customers (note that the minimal value of the server span is $J - 1$).  Figure~\ref{fig.Q_hists} shows the distribution of the server span for the three systems under the same four service time distributions, as estimated from simulation.  Clearly, the distribution in each system is sensitive to the service time distribution.

\begin{figure}[htp]
\centering
\includegraphics[width=\textwidth]{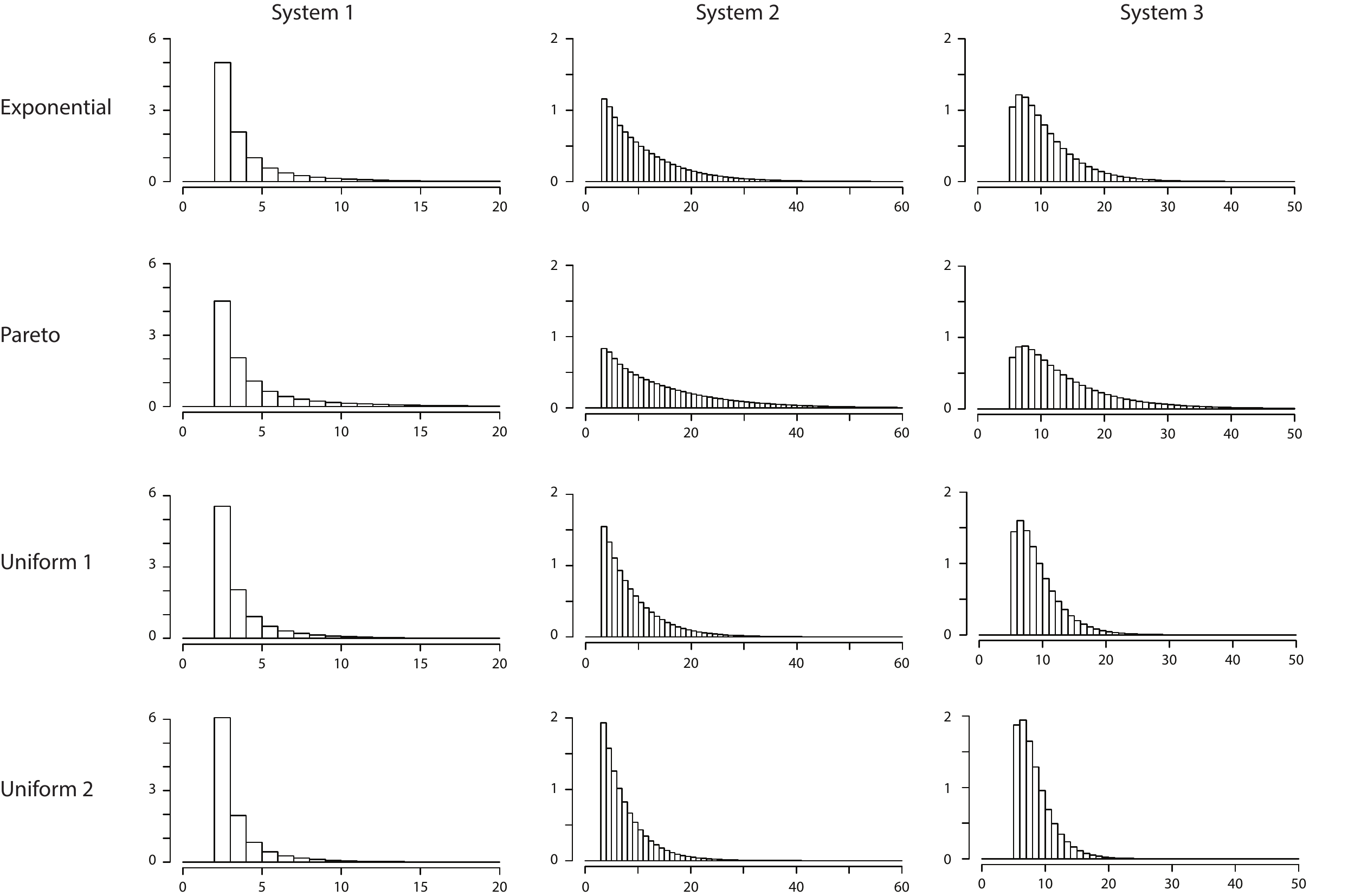}
\caption{Server span distribution for the three systems under four service time distributions.  Counts on the vertical axes are $\times10^7$.}\label{fig.Q_hists}
\end{figure}

From \cite{adan-busic-mairesse-weiss:15}, the steady-state distribution of the server permutations in the exponential case is given by
\begin{equation*}
\pi_R(S_1,\ldots,S_J) = B^s\prod_{\ell=1}^{J-1}(\beta_{\{S_1,\ldots,S_\ell\}} - \alpha_{\U\{S_1,\ldots,S_\ell\}})^{-1},
\end{equation*}
where $B^s$ is a normalizing factor.  This distribution was estimated by simulation also for the non-exponential cases, and the results for systems~1 and~2 are shown in Tables~\ref{tab.permutations.1} and~\ref{tab.permutations.2} (the results for system~3 are omitted due to the size of the table --- $6! = 720$ rows).  The deviations of the estimated values from the theoretical ones are small, but statistically significant: when using again Hotelling's $T^2$ test, the $p$-values in all 6 cases (2 systems $\times$ 3 distributions) was $< 10^{-15}$.  In contrast, the $p$-values for systems~1 and~2 under simulated exponential service times were 0.372 and 0.443, respectively.  Thus, the steady-state distribution of the server permutations is also sensitive to the service time distribution.

\begin{table}[h]
\centering
\begin{tabular}{*{7}{c}}
\toprule
permutation & \begin{tabular}[c]{@{}c@{}}theoretical\\(exponential)\end{tabular}& Pareto	& Uniform~1 	& Uniform~2 \\
\midrule
1-2-3	& .1	& .1018	& .0973	& .0912 \\
1-3-2	& .2	& .1996	& .2006	& .2010 \\
2-1-3	& .2	& .1988	& .2021	& .2078 \\
2-3-1	& .2	& .1988	& .2021	& .2078 \\
3-1-2	& .2	& .1993	& .2006	& .2010 \\
3-2-1	& .1	& .1017	& .0974	& .0912 \\
\bottomrule
\end{tabular}
\caption{Steady-state distribution of server permutations, system~1.}
\label{tab.permutations.1}
 \end{table}

\begin{table}[h]
\centering
\begin{tabular}{*{7}{c}}
\toprule
permutation & \begin{tabular}[c]{@{}c@{}}theoretical\\(exponential)\end{tabular}& Pareto	& Uniform~1 	& Uniform~2 \\
\midrule
1-2-3-4	& .0232	& .0289	& .0194	& .0169 \\
1-2-4-3	& .0077	& .0089	& .0068	& .0061 \\
1-3-2-4	& .0116	& .0133	& .0106	& .0103 \\
1-3-4-2	& .0023	& .0026	& .0022	& .0022 \\
1-4-2-3	& .0058	& .0083	& .0045	& .0038 \\
1-4-3-2	& .0035	& .005	& .0027	& .0024 \\
2-1-3-4	& .0310	& .0286	& .0319	& .0323 \\
2-1-4-3	& .0103	& .0092	& .0110	& .0116 \\
2-3-1-4	& .0929	& .0933	& .0888	& .0839 \\
2-3-4-1	& .0929	& .1042	& .0849	& .0792 \\
2-4-1-3	& .0077	& .0064	& .0088	& .0096 \\
2-4-3-1	& .0232	& .0215	& .0243	& .0250 \\
3-1-2-4	& .0232	& .0201	& .0251	& .0257 \\
3-1-4-2	& .0046	& .0039	& .0052	& .0054 \\
3-2-1-4	& .1394	& .1319	& .1451	& .1499 \\
3-2-4-1	& .1394	& .1426	& .1416	& .1456 \\
3-4-1-2	& .0139	& .0120	& .0150	& .0153 \\
3-4-2-1	& .0697	& .0680	& .0695	& .0690 \\
4-1-2-3	& .0232	& .0231	& .0223	& .0209 \\
4-1-3-2	& .0139	& .0137	& .0137	& .0137 \\
4-2-1-3	& .0232	& .0199	& .0260	& .0276 \\
4-2-3-1	& .0697	& .0686	& .0692	& .0671 \\
4-3-1-2	& .0279	& .0248	& .0301	& .0313 \\
4-3-2-1	& .1394	& .1411	& .1410	& .1451 \\
\bottomrule
\end{tabular}
\caption{Steady-state distribution of server permutations, system~2.}
\label{tab.permutations.2}
 \end{table}


\end{document}